\documentclass{article}
\usepackage[english]{babel}
\usepackage{amsthm}
\usepackage{comment}
\usepackage[utf8]{inputenc}
\usepackage{float}
\usepackage{amssymb}
\usepackage{amsmath}
\usepackage[T1]{fontenc}
\usepackage[colorlinks]{hyperref}
\usepackage{amsxtra}
\usepackage{amstext}
\usepackage{amssymb}
\usepackage{amsthm}
\usepackage{graphicx}
\usepackage{amsfonts}
\usepackage{multicol}
\usepackage{setspace}
\usepackage[usestackEOL]{stackengine}
\usepackage[noabbrev]{cleveref}
\usepackage{MnSymbol}
\usepackage{soul}

\newtheorem{theorem}{Theorem}[section]
\newtheorem{remark}{Remark}[section]
\newtheorem{definition}{Definition}[section]
\newtheorem{proposition}{Proposition}[section]
\newtheorem{lemma}{Lemma}[section]
\newtheorem{claim}{Claim}[section]

\usepackage[utf8]{inputenc}
\usepackage{mathtools}

\DeclarePairedDelimiter\ceil{\lceil}{\rceil}
\DeclarePairedDelimiter\floor{\lfloor}{\rfloor}

\title{Large Deviations Principles for Coulomb gases at intermediate temperature regime}
\author{David Padilla-Garza}

\begin{document}

\maketitle

\begin{abstract}
    This paper deals with Coulomb gases at an intermediate temperature regime. We define a local empirical measure and identify a critical temperature scaling. We show that if the scaling of the temperature is supercritical, the local empirical measure satisfies an LDP with an entropy-based rate function. We also show that if the scaling of the temperature is subcritical, the local empirical measure satisfies an LDP with an energy-based rate function. In the critical temperature scaling regime, we derive an LDP-type result in which the "rate function" features the competition of entropy and energy terms. 
\end{abstract}

\section{Introduction}

Coulomb gases are a system of particles of the same charge that interact via a repulsive kernel, and are confined by an external potential. Let $X_{N}=(x_{1}, x_{2},...x_{N})$ with $x_{i} \in \mathbf{R}^{d}$ and let 
\begin{equation}
    \mathcal{H}_{N}\left( X_{N}\right)=\frac{1}{2}\sum_{i\neq j}g\left( x_{i}-x_{j} \right)+N \sum_{i=1}^{N} V\left( x_{i} \right),
\end{equation}
where
\begin{equation}
\begin{cases}
g(x)=\frac{1}{|x|^{d-2}} \text{ if } d \geq 3 \\
g(x)=-\log(|x|) \text{ if } d = 1,2
\end{cases}
\end{equation}
is the Coulomb kernel for $d \geq 2$, i.e. $g$ satisfies for $d \geq 2$,
\begin{equation}
    \Delta g = c_{d}\delta_{0},
\end{equation} 
where $c_{d}$ is a constant that depends only on $d$. Often, Coulomb gases at non-zero temperature are considered, these are modeled by a point process whose density is given by the Gibbs measure associated to the Hamiltonian:
\begin{equation}\label{Gibbsmeasure}
    d\mathbf{P}_{N,\beta}=\frac{1}{Z_{{N}, \beta}} \exp\left( -  \beta \mathcal{H}_{N} \right) d X_{N},
\end{equation}
where
\begin{equation}
    Z_{N, \beta} = \int_{(\mathbf{R}^{d})^{N}}  \exp\left( - \beta \mathcal{H}_{N} \right) d\, X_{N}.
\end{equation}
In this notation $\beta$ is the inverse temperature (which may depend on $N$).

As long as $\frac{1}{N} \ll \beta,$ we have that the empirical measure
\begin{equation}
    {\rm emp}_{N} := \frac{1}{N} \sum_{i=1}^{N} \delta_{x_{i}}
\end{equation}
converges (weakly in the sense of probability measures) almost surely under the Gibbs measure to $\mu_{V},$ where $\mu_{V}$ is the minimizer of the mean-field limit
\begin{equation}
\label{eq:meanfieldlimit}
    \mathcal{I}_{V}\left( \mu \right) = \frac{1}{2} \iint_{\mathbf{R}^{d} \times \mathbf{R}^{d}} g(x-y) d  \mu(x) \, d \mu(y) + \int_{\mathbf{R}^{d}} V \, d \mu
\end{equation}
among probability measures. 

Coulomb gases have a wide range of applications in Statistical Mechanics and Random Matrix Theory, among other areas, see \cite{borodin2019random, serfaty2017microscopic}  for a more in-depth discussion.

The most fundamental LDP in log gases is found in \cite{arous1997large}. Adapting their results to our setting, it was proved that in the regime $\beta=1$ and $d=1,$ the push-forward of $\mathbf{P}_{N, \beta}$ by ${\rm emp}_{N}$ satisfies an LDP { at speed $N^{2} $} with rate function given by 
\begin{equation}
    \mathcal{F}(\mu) =  \mathcal{I}_{V}\left( \mu \right) -  \mathcal{I}_{V}\left( \mu_{V} \right).
\end{equation}
This result was originally motivated by Hermitian Random Matrix Theory. 

An analogous statement in dimension $2$ was proved in \cite{ben1998large}. In \cite{chafai2014first} the authors deal with a general repulsive interaction $g$ in dimension $d \geq 1.$ In \cite{petz1998logarithmic}, the authors derive an LDP for the eigenvalues of some non-symmetric random matrices. 

As mentioned before, the regime $\beta = \frac{C}{N}$ is substantially different since the empirical measure does not converge to the equilibrium measure. We call this the high temperature regime. Nevertheless, it is possible to identify the limit of the empirical measures as the thermal equilibrium measure:
\begin{equation}
    \mu_{\beta} = \operatorname{argmin}_{\mu}  \mathcal{I}_{V}\left( \mu \right) + \frac{1}{C} \mbox{ent}[\mu],
\end{equation}
where ${\rm ent}$ is given by Definition \ref{def:ent}, and the minimum is taken over probablity measures. Moreover, the push-forward of the Gibbs measure $\mathbf{P}_{N, \beta}$ by the empirical measure satisfies an LDP { at speed $CN$} with rate function
\begin{equation}
     \mathcal{F}(\mu) = \mathcal{I}_{V}\left( \mu \right) + \frac{1}{C} \rm{ent}[\mu] - \left(\mathcal{I}_{V}\left( \mu_{\beta} \right) + \frac{1}{C} \rm{ent}[\mu_{\beta}]  \right).
\end{equation}
This result can be found in \cite{garcia2019large}, which also treats Coulomb gases on compact manifolds.

In our setting, the intensity and sign of the charge of the particles are fixed; in reference, \cite{bodineau1999stationary}, however, the authors also consider the case of the intensity of the charges being a random variable which takes values in $\{-1, 1\}$. Having positive and negative charges implies that there are attractive interactions, which are harder to deal with. 

A widely studied question in Coulomb gases is that of the fluctuations of the difference between ${\rm emp}_{N}$ and $\mu_{V}.$ In order to understand these fluctuations, it is convenient to multiply this difference by a test function $\varphi,$ the resulting object is called the first order statistic:
\begin{equation}
  \mbox{Fluct}_{N}(\varphi) = N \int \varphi \, d\left( {\rm emp}_{N} - \mu_{V} \right).
\end{equation}
In \cite{leble2018fluctuations} it was proved that in two dimensions (under mild technical additional conditions) $\mbox{Fluct}_{N}(\varphi)$ converges in law to a Gaussian random variable with mean 
\begin{equation}
\mbox{mean} =  \frac{1}{2 \pi} \left( \frac{1}{\beta} - \frac{1}{4} \right) \int_{\mathbf{R}^{2}} \Delta \varphi \left( \mathbf{1}_{\Sigma} + (\log \Delta V)^{\Sigma} \right)
\end{equation}
and variance
\begin{equation}
    \mbox{Var} = \frac{1}{ 2 \pi \beta } \int_{\mathbf{R}^{2}} |\nabla \varphi^{\Sigma}|^{2}.
\end{equation}
In this notation, $\Sigma$ is the support of the equilibrium measure, and $g^{\Sigma}$ is the harmonic extension of $g$ outside $\Sigma,$ i.e. the only continuous function which agrees with $g$ in $\Sigma$ up to the boundary and is harmonic and bounded in $\mathbf{R}^{2} \setminus \Sigma.$ A related, very similar result was obtained simultaneously in \cite{bauerschmidt2019two}. Analogous results were obtained in one dimension in \cite{bekerman2018clt} and \cite{lambert2019quantitative}, generalizing the work of \cite{johansson1998fluctuations}, \cite{shcherbina2013fluctuations}, and \cite{borot2013asymptotic}. In \cite{bauerschmidt2017local}, the authors derive local laws and moderate deviation bounds. In \cite{hardy2021clt}, the authors derive a CLT for linear statistics of $\beta-$ensembles at high temperature, in this reference, the authors also derive concentration inequalities. In \cite{serfaty2020gaussian}, the author deals with linear statistics replacing $\mu_{V}$ with the thermal equilibrium measure.  All of the references just mentioned, except \cite{serfaty2020gaussian} and \cite{hardy2021clt} deal with $\beta$ proportional to $N^{1 - \frac{2}{d}}$ for $d \geq 2$ (in our notation), or $\beta$ constant in $d=1$. This paper wanders into the mainly unexplored territory of Coulomb gases at general temperature regimes and high dimensions. 

Coulomb gases are also widely studied due to their connection with Random Matrix Theory, see \cite{bourgade2014universality, bourgade2014local, bourgade2014edge, bourgade2012bulkb}for recent developments. Most of the problems studied in connection with Random Matrix Theory are in dimensions $1$ and $2$, therefore the result in this paper (which holds only in dimensions $3$ and higher) is not applicable to that setting. However, the main result in this paper is applicable to quaternionic Gaussian ensembles.  

Since the equilibrium measure typically has compact support,  there are $N$ particles in a bounded domain in $\mathbf{R}^{d},$ and so, typically the particles are at distance $N^{-\frac{1}{d}}$ of each other. After applying a dilation of magnitude $N^{\frac{1}{d}}$ to Euclidean space, one observes individual particles. An LDP { at speed $N$} for Coulomb gases at this scale was obtained in \cite{leble2017large1}, and the rate function combines two terms: one comes from the Hamiltonian and the other one is related to entropy. Similar results were obtained in \cite{hardin2018large} for hyper-singular Riesz gases, and in \cite{leble2017large2} for two-component plasmas.

Details of the convergence of ${\rm emp}_{N}$ to $\mu_{V}$ were obtained in \cite{chafai2018concentration}. In this reference, the authors also study the relation between the electric energy and norms on probability measures.
One of their results concerning the convergence of ${\rm emp}_{N}$ to $\mu_{V}$ is the following: 
If $\beta>0$, then under mild additional assumptions, there exist
constants $u,v > 0$ depending on $\beta$ and $V$ only such that, for any $N \geq 2$ and
\begin{equation}
    r \geq \begin{cases}
    v \sqrt{\frac{\log N}{N}} \mbox{ if } d=2\\
    v N^{-\frac{1}{d}} \mbox{ if } d>2
    \end{cases}
\end{equation}
we have
\begin{equation}
    \mathbf{P}_{N,\beta} (W_{1}(\mu, \nu) \geq r) \leq \exp(- u N^{2} r^{2}),
\end{equation}
where $W_{1}$ is the Wasserstein distance (see \cite{chafai2018concentration}).

The main contribution of this paper is to clarify the
relationship between temperature scales and length scales for the mesoscopic behavior of particle systems with Coulomb interactions. The way to do this is to look at rare events at a mesoscale and understand them by means of an LDP. An important idea in this work is to exploit the different scaling relations satisfied by the Coulomb energy and the entropy. This work also exploits the smearing technique, used for example in \cite{chafai2018concentration, leble2017large1,rougerie2016higher,sandier20152d}. A large part of this work is devoted to simplifying expressions for partition functions, derived via a variational characterization. 

At a macroscopic scale, it has been well-known that the temperature and the energy compete if the energy is of order $\beta = \frac{1}{N}$, in the sense that the empirical field (the macroscopic observable) satisfies an LDP that involves energy and entropy terms. At a microscopic scale, it was recently proved in \cite{leble2017large1} that the energy and the entropy compete if the temperature is of order $\beta = N^{\frac{2-d}{d}}$ (for $d \geq 3$), in the sense that the tagged empirical field (the microscopic observable) satisfies an LDP that involves energy and entropy terms. This raises the natural question: given a length scale between the macroscopic and microscopic, is there a temperature regime in which the temperature and entropy compete (in the sense that there is an LDP containing energy and entropy terms)? This paper answers this question. 

Given a length scale, we will identify a critical temperature regime. Of course, this problem is equivalent to identifying a critical length scale given a temperature regime. This last approach was taken in \cite{armstrong2019local}. However, despite analyzing the interplay between temperature and length scale, this work and \cite{armstrong2019local} are pretty much independent. \cite{armstrong2019local} deals with the tagged empirical field: an observable obtained by averaging the empirical field over a certain region. This observable is fundamentally different from the local empirical measure. Furthermore, the $\rho_{\beta}$ identified in \cite{armstrong2019local} does not coincide with the "critical length scale" identified in our work. The main results in this work and in \cite{armstrong2019local} are independent:  our results do not follow from \cite{armstrong2019local}, and the results in \cite{armstrong2019local} do not follow from our results. The techniques used are also fundamentally different. This work is not an attempt to prove any conjecture in \cite{armstrong2019local}. 

A significant part of this work is devoted to computing, with high precision, certain partition functions. This is similar to  obtaining a Laplace principle, as in \cite{garcia2019large}. However, the techniques in \cite{garcia2019large} would only allow us to obtain the leading order term in the partition function. This would be greatly insufficient to conclude, and so it is necessary to take a different approach in order to identify the next-order terms. 

\section{Main definitions and statement of main results}

This section defines the most important objects for the rest of the paper and states the main results.

We begin with definitions related to the empirical measure. 

\begin{definition}
Given $X_{N} \in \mathbf{R}^{d \times N},$ with 
\begin{equation}
    X_{N} = (x_{1}, ...x_{N}),
\end{equation}
we denote
\begin{equation}
    {\rm emp}_{N}(X_{N}) = \frac{1}{N} \sum_{i=1}^{N} \delta_{x_{i}}.
\end{equation}
In order to make the notation more clear, we will often write 
\begin{equation}
    {\rm emp}_{N}
\end{equation}
instead of
\begin{equation}
    {\rm emp}_{N}(X_{N}).
\end{equation}

Given $x \in \mathbf{R}^{d}$ and $R \in \mathbf{R}^{+}$ we denote by 
\begin{equation}
    \square(x,R)=\left(-\frac{R}{2},\frac{R}{2} \right)^{d}+x.
\end{equation}
We will also use the notation
\begin{equation}
    \square_{R}=\square(0,R).
\end{equation}
Let 
\begin{equation}
    x_{i}^{\lambda}=N^{\lambda}x_{i}.
\end{equation}
We now define the main observable of this paper: the \emph{local empirical measure}
\begin{equation}
    {\rm lemp}_{N}^{\lambda}(X_{N})=\frac{1}{N^{1-\lambda d}} \sum_{i=1}^{N} \delta_{x_{i}^{\lambda}}|_{\square_{R}}.
\end{equation}
Even though  ${\rm lemp}_{N}^{\lambda}(X_{N})$ depends on $\lambda$ and $X_{N}$, we will sometimes omit this dependence for ease of notation and simply write ${\rm lemp}_{N}$. Note that ${\rm lemp}_{N}$ is a measure with support contained in $\square_{R},$ and with mass which we expect remains bounded if $X_{N}$ is distributed according to $\mathbf{P}_{N, \beta}$.
\end{definition}

This paper deals with the empirical measure at a mesoscopic scale, i.e. at a scale $N^{-\lambda},$ where 
\begin{equation}
   \lambda \in \left(0,\frac{1}{d}\right). 
\end{equation}
We choose the name mesoscopic because the scale $\lambda=0$ is macroscopic, while the scale $\lambda = \frac{1}{d}$ is microscopic, i.e. a scale which is of the same order of magnitude as the distance between particles. Without loss of generality, we assume that we blow up around the origin. For the general case, we may simply consider a modified potential. 

The idea is to define a mesoscopic observable. This definition is inspired by interpolating between the empirical measure: $\frac{1}{N}\sum_{i=1}^{N} \delta_{x_{i}}$, and the empirical field: $\sum_{i=1}^{N} \delta_{x_{i}^{N^{\frac{1}{d}}}}$. The factor of $N^{\lambda}$ in the dilation corresponds to a mesoscale, while the normalizing factor of $N^{1 - \lambda d}$ is necessary to obtain a bounded nonzero quantity. Note that the local empirical measure is, in general, not a probability measure but only a positive measure. The local empirical measure is more similar to the empirical measure in the sense that it converges to a continuous measure. 

We now define the most basic functionals used in the paper: energy and entropy. We also define a few modifications of the functionals which will be used throughout the paper. 

\begin{definition}
\label{def:ent}
Given two measures $\mu$ and $\rho$, the relative entropy ${\rm ent}[\mu | \rho ]$ is defined as 
\begin{equation}
{\rm ent}[\mu | \rho ]=
    \begin{cases}
    \int  \log \left( \frac{d\mu}{d\rho} \right) \, d{\mu} \quad  if \mu \ll \rho \\
    \infty\quad  \text{ o.w.}
    \end{cases}
\end{equation}

The entropy of a measure ${\rm ent}[\mu]$ is defined as ${\rm ent}[\mu|\mathcal{L}]$, where $\mathcal{L}$ denotes the Lebesgue measure on $\mathbf{R}^{d}.$

In the remainder of the paper, we commit the abuse of notation of not distinguishing between a measure and its density. 
\end{definition}

\begin{definition}
The electric energy of a measure $\nu$ is defined as 
\begin{equation}
    \mathcal{E}(\nu)=\iint_{\mathbf{R}^{d} \times \mathbf{R}^{d}} g(x-y) d  \nu(x) \, d \nu(y),
\end{equation}
and
\begin{equation}
    \mathcal{E}^{\neq}(\nu)=\iint_{\mathbf{R}^{d} \times \mathbf{R}^{d} \setminus \Delta} g(x-y) d  \nu(x) \, d \nu(y),
\end{equation}
where
\begin{equation}
    \Delta = \{ (x,x) \in \mathbf{R}^{d} \times \mathbf{R}^{d} \}.
\end{equation}
Given a measurable set $\Omega \subset \mathbf{R}^{d}$, we will also use the notation 
\begin{equation}
    \mathcal{E}^{\neq}_{\Omega}(\nu)=\iint_{\mathbf{R}^{d} \times \mathbf{R}^{d} \setminus \Delta_{\Omega}} g(x-y) d  \nu(x) \, d \nu(y),
\end{equation}
where
\begin{equation}
    \Delta_{\Omega} = \{(x,x) \in \Omega \times \Omega \}.
\end{equation}
\end{definition}

\begin{definition}
We define the free energy of a measure as 
  \begin{equation}
     \mathcal{E}_{\beta}\left( \mu \right)=  \mathcal{I}_{V}\left( \mu \right)+\frac{1}{N\beta}\int_{\mathbf{R}^{d}} \mu \log \left( \mu \right).
 \end{equation}
 
We define the thermal equilibrium measure $\mu_{\beta}$ as 
\begin{equation}
     \mu_{\beta}=\operatorname{argmin}_{\mu \in \mathcal{P}(\mathbf{R}^{d})} \mathcal{E}_{\beta}(\mu),
\end{equation}
where $\mathcal{P}(\mathbf{R}^{d})$ denotes the set of probability measures on $\mathbf{R}^{d}$. More generally, we will use the notation $\mathcal{P}(\Omega)$ for the set of probability measures on $\Omega \subset \mathbf{R}^{d}$.  

For existence, uniqueness and basic properties of $\mu_{\beta}$ see \cite{armstrong2019thermal}. 

We also define the equilibrium measure $\mu_{V}$ by 
\begin{equation}
     \mu_{V}=\operatorname{argmin}_{\mu \in \mathcal{P}(\mathbf{R}^{d})} \mathcal{I}_{V}\left( \mu \right),
\end{equation}
where $\mathcal{I}_{V}$ is given by equation \eqref{eq:meanfieldlimit}. 

For existence, uniqueness and basic properties of $\mu_{V}$ see for example \cite{serfaty2018systems} and references therein. 
\end{definition}

We proceed with a few definitions regarding measures. 

\begin{definition}
Given a measurable set $\Omega \subset \mathbf{R}^{d}$, we denote $\mathcal{M}(\Omega)$ the space of measures on $\Omega$ which are either of bounded variation or have a definite sign. We also define, for any $\mu \in \mathcal{M}(\Omega)$ 
\begin{equation}
    |\mu| = \mu (\Omega).
\end{equation}
\end{definition}

\begin{definition}
Given a measurable set $\Omega$, and a measure $\mu$ on $\Omega$, we define the bounded Lipschitz norm of $\mu$, denoted $\|\mu\|_{BL}$ as 
\begin{equation}
    \|\mu\|_{BL} = \sup_{f \in {\rm Lip}_{1}(\Omega)} \int_{\mathbf{R}^{d}} f \, d\mu,
\end{equation}
where ${\rm Lip}_{1}(\Omega)$ denotes the set of Lipschitz functions on $\Omega$ whose absolute value is bounded by $1$, and Lipschitz constant is also smaller than $1$. 

Unless otherwise specified, any distance between measures will refer to the bounded Lipschitz norm. In particular, given $\epsilon>0$ we define 
\begin{equation}
    B(\nu, \epsilon) = \{ \mu \in \mathcal{M}(\Omega) | \| \mu - \nu \|_{BL} \leq \epsilon \}.
\end{equation}

We recall that the bounded Lipschitz norm metricizes the topology of weak convergence. 
\end{definition}

\begin{definition}
Let $\mu,\nu \in \mathcal{M}^{+}(\Omega),$ we define
\begin{equation}
    \mathcal{N}[\mu|\nu] = {\rm ent}[\mu| \nu] + |\nu| - |\mu|,
\end{equation}
where $\mathcal{M}^{+}(\Omega)$ denotes the set of positive measures on a set $\Omega$.
\end{definition}

We will now introduce the rate functions for the different LDP's. These rate functions are based on the entropy functional, the energy functional, or both.

\begin{definition}
Given a domain $\Omega \subset \mathbf{R}^{d}$ and a scalar $\alpha \in \mathbf{R}^{+},$ we define the function $\Phi^{\alpha}_{\Omega},$ defined for an absolutely continuous measure $\mu$ on $\Omega$ as 
\begin{equation}
   \Phi^{\alpha}_{\Omega}(\mu)= \inf_{\rho: \mathbf{R}^{d} \setminus \Omega \to \mathbf{R}^{+}} \iint_{\mathbf{R}^{d} \times \mathbf{R}^{d}} g(x-y) ( \mu(x)+\rho(x) - \alpha)dx ( \mu(y)+\rho(y) - \alpha) dy.
\end{equation}
\end{definition}

\begin{definition}
\label{defT}
Given a measure $\nu$, we define the measure $\nu^{\tau}$ as 
\begin{equation}
    \nu^{\tau}(\Omega) = \tau^{d} \nu (\tau^{-1} \Omega).
\end{equation}
For a measure $\mu$ defined on $\Omega,$ we denote by 
\begin{equation}
\label{defofTN}
    \mathbf{T}_{\lambda}^{N}(\nu) = \inf_{\rho \in \mathcal{M}^{+} \mathbf{R}^{d} \setminus \Omega} \left( \frac{1}{2}\mathcal{E}\left( \nu + \rho - \mu_{\beta}^{N^{\lambda}} \right) - \int_{\mathbf{R}^{d}} \log\left( \mu_{\beta}^{N^{\lambda}} \right) d\rho + {\rm ent}[\rho] \right),
\end{equation}
where the infimum is taken over $\rho$ such that
\begin{equation}
    \int_{\mathbf{R}^{d}} \nu + \rho - \mu_{\beta}^{N^{\lambda}}=0.
\end{equation}

 We also define
\begin{equation}
    \mathcal{T}_{\lambda}^{N}(\nu) =  \mathbf{T}_{\lambda}^{N}(\nu) + {\rm ent}[\nu| \mu_{V}(0)\mathbf{1}_{\Omega}].
\end{equation}
\end{definition}

In this paper, we deal with general a general potential $V$. We only impose some regularity and growth conditions, which we make precise in the next definition.

\begin{definition}
We call a potential $V: \mathbf{R}^{d} \to \mathbf{R}$, with $d \geq 3$ {admissible} if:
\begin{itemize}
    \item[1.]  $V \in C^{2}$.
    
    \item[2.] 
    \begin{equation}
    \lim_{x \to \infty} V(x) = \infty.    
    \end{equation}
    
    \item[3.] 
    \begin{equation}
    \int_{|x| > 1} \exp \left( -\alpha V (x) \right) \, dx < \infty    
    \end{equation}
    for all $\alpha>0$. 
    
    \item[4.] 
    \begin{equation}
    \Delta V \geq \alpha >0    
    \end{equation}
    for some $\alpha$, in a neighborhood of $\Sigma$, defined as the support of $\mu_{V}$. 
    
    \item[5.] $V|_{\Sigma} \in C^{\infty}(\Sigma)$. 
    
    \item[6.] $0 \in {\rm int}(\Sigma)$. 
\end{itemize}
\end{definition}

\begin{remark}
    If $V$ is admissible, the equilibrium measure $\mu_{V}$ is bounded and has compact support, see \cite{serfaty2018systems}.
\end{remark}

Finally, before stating the main result, we recall the definition of rate function and Large Deviations Principle (LDP). 

\begin{definition}
(Rate function) Let $X$ be a metric space (or a topological space). A rate function is a l.s.c. function $I:X \to [0,\infty],$ it is called a good rate function if its sublevel sets are compact.
\end{definition}

\begin{definition}[LDP]
 Let $P_{N}$ be a sequence of Borel probability measures on $X$ and $a_{N}$ a sequence of positive reals such that $a_{N} \to \infty.$ Let $I$ be a good rate function on $X.$ The sequence $P_{N}$ is said to satisfy a Large Deviations Principle (LDP) at speed $a_{N}$ with (good) rate function $I$ if for every Borel set $E \subset X$ the following inequalities hold:
\begin{equation}
    - \inf_{E^{\mathrm{o}}} I \leq \liminf_{N \to \infty} \frac{1}{a_{N}} \log \left( P_{N}(E) \right) \leq \limsup_{N \to \infty} \frac{1}{a_{N}} \log \left( P_{N}(E) \right) \leq  - \inf_{\overline{E}} I,
\end{equation}
where $E^{\mathrm{o}}$ and $\overline{E}$ denote respectively the interior and the closure of a set $E.$
Formally, this means that $P_{N}(E) \simeq \exp(-a_{N} \inf_{{E}} I).$
\end{definition}

The main result of this paper is the following theorem:
\begin{theorem}\label{maintheorem}
Assume that $d \geq 3$ and the potential $V$ are admissible. Let $\beta = N^{-\gamma}$ with $\gamma \in (\frac{d-2}{d}, 1).$ Assume that $\mu_{V}$ is bounded away from $0$ inside its support. Let 
\begin{equation}
    \gamma^{*}=1-2\lambda,
\end{equation}
and assume that
\begin{equation}
    \lambda < \frac{1}{d(d+2)}.
\end{equation}
Then:
\begin{itemize}
      \item[$\bullet$] If $\gamma<\gamma^{*}$ (subcritical regime) then the push-forward of $\mathbf{P}_{N,\beta}$ by ${\rm lemp}_{N}$ satisfies an LDP in the topology of weak convergence at speed $\beta N^{2-(d+2)\lambda}$ and rate function
    \begin{equation}
        \frac{1}{2}\Phi^{\mu_{V}(0)}_{\square_{R}}(\mu).
    \end{equation}
    \item[$\bullet$] If $\gamma>\gamma^{*}$ (supercritical regime) then the push-forward of $\mathbf{P}_{N, \beta}$ by ${\rm lemp}_{N}$ satisfies an LDP in the topology of weak convergence at speed $N^{1-\lambda d}$ and rate function
    \begin{equation}
        \mathcal{N}[\mu|\mu_{V}(0) \mathbf{1}_{\square_{R}}] .
    \end{equation}
     \item[$\bullet$] If $\gamma=\gamma^{*}$ (critical regime) and $\nu \in L^{\infty}$ then
     \begin{equation}
         \lim_{\epsilon \to 0} \limsup_{N \to \infty} \left( \frac{1}{\beta N^{2 - \lambda(d+2)} } \log \left( \mathbf{P}_{N,\beta} ({\rm lemp}_{N} \in {B}(\nu, \epsilon)) \right) + \mathcal{T}^{N}_{\lambda} (\nu) \right) = 0.
     \end{equation}
     Similarly,
     \begin{equation}
         \lim_{\epsilon \to 0} \liminf_{N \to \infty} \left( \frac{1}{\beta N^{2 - \lambda(d+2)} } \log \left( \mathbf{P}_{N,\beta} ({\rm lemp}_{N} \in  {B}(\nu, \epsilon)) \right) + \mathcal{T}^{N}_{\lambda} (\nu) \right) = 0.
     \end{equation}
\end{itemize}
\end{theorem}

\begin{remark}
    The rate functions have the same minimizer in all cases: $\mu_{V}(0) \mathbf{1}_{\square_{R}}$. This is clear, since in this temperature regime, the empirical measure concentrates on the thermal equilibrium measure for all scales larger than $N^{-\frac{1}{d}}$, as was proved in \cite{padilla2023concentration}. The typical event, therefore, is trivial; and it is a rare event that deserves to be looked at. 
\end{remark}

\begin{remark}
    Nearly all hypotheses in Theorem \ref{maintheorem} are essential. The hypothesis that $\lambda < \frac{1}{d(d+2)}$, however is not. It is an artifact of the proof, and it is needed to bound a specific error term. Bounding this error term is necessary if one uses the regularization procedure, i.e. it is needed to bound the difference between the energy of a discrete probability measure and a continuous one. This technique is very common in the field. We expect that a similar result will be true for $\lambda \geq \frac{1}{d(d+2)}$, but proving this would require essentially different techniques. 
\end{remark}

\begin{remark}
It is natural to ask if there is an analog of Theorem \ref{maintheorem} in the extreme cases 
\begin{equation}
    \beta=N^{-1}, \lambda=0 \quad\text{and} \quad \beta=N^{-\frac{d-2}{d}}, \lambda=\frac{1}{d}.
\end{equation}
In the case 
\begin{equation}
    \beta=N^{-1}, \lambda=0, 
\end{equation}
Theorem \ref{maintheorem} has a very natural generalization, as mentioned in the introduction. It was proved in \cite{garcia2019large} that ${\rm emp}_{N}$ satisfies an LDP at speed $N$ with rate function 
\begin{equation}
     \mathcal{F}(\mu) = \mathcal{I}_{V}\left( \mu \right) + \rm{ent}[\mu] - \left(\mathcal{I}_{V}\left( \mu_{\beta} \right) + \rm{ent}[\mu_{\beta}]  \right).
\end{equation}

The case 
\begin{equation}
 \beta=N^{-\frac{d-2}{d}}, \lambda=\frac{1}{d}
\end{equation}
is substantially different, because at a microscopic scale, we do not observe a continuous distribution but rather individual particles. A similar problem was treated in \cite{leble2017large1}. Even though the result is substantially different, it has a similar flavor, since the authors prove an LDP { at speed $N$} in which the rate function contains the sum of an entropy term and an electric energy term. 
\end{remark}

\section{Additional definitions}

We proceed with a few additional definitions related to the empirical measure.
\begin{definition}
Let $R \in \mathbf{R}^{+}$ be fixed, and $X_{N} \in \mathbf{R}^{d \times N}.$ We define $y_{i}, z_{j}$ such that
\begin{equation}
    X_{N}=(y_{1}, y_{2}...y_{i_{N}}, z_{1}, z_{2},...z_{j_{N}}),
\end{equation}
where
\begin{equation}
    y_{k} \in \square_{\frac{R}{N^{\lambda}}}, \quad z_{k} \notin \square_{\frac{R}{N^{\lambda}}},
\end{equation}
and 
\begin{equation}
    i_{N}+j_{N}=N.
\end{equation}
Let 
\begin{equation}
    Y_{N}=(y_{1},...y_{i_{N}})
\end{equation}
and
\begin{equation}
    {\rm emp}_{N}'(Y_{N})=\frac{1}{N} \sum_{k=1}^{i_{N}} \delta_{y_{k}}.
\end{equation}
Similarly, let
\begin{equation}
    Z_{N}=(z_{1},...z_{i_{N}}).
\end{equation}

\end{definition}

\begin{definition}
\label{addenergydef}
Given an integer $M,$ and $\epsilon>0,$ we denote by $\mathcal{A}_{M}^{\epsilon}(\Omega)$ the set of measures which are purely atomic with weight $\epsilon,$ i.e.
\begin{equation}
    \mathcal{A}_{M}^{\epsilon}(\Omega)=\{ \mu \in \mathcal{M}^{+}(\Omega) | \mu = \epsilon \sum_{i=1}^{M} \delta_{x_{i}}\}.
\end{equation}

Given a measure $\mu \in \mathcal{M}^{+}(\mathbf{R}^{d}),$ an integer $M,$ a region $\Omega \subset \mathbf{R}^{d}$ and $\epsilon>0,$ we define
\begin{equation}\label{definitiongeneralPhi}
    \mathbf{W}_{\Omega, \epsilon}^{M, \mu}(\nu)=\inf_{\rho \in \mathcal{A}_{M}^{\epsilon}(\mathbf{R}^{d} \setminus \Omega)} \mathcal{E}^{\neq}(\nu + \rho - \mu),
\end{equation}
where $\nu \in \mathcal{M}^{+}(\Omega).$

We also define
\begin{equation}\label{defofT}
    \mathbf{T}_{\lambda}^{N,\neq} (\nu) =\\
    \inf_{\rho \in \mathcal{M}^{+} \mathbf{R}^{d} \setminus \Omega} \Bigg( \frac{1}{2}\mathcal{E}^{\neq}_{\square_{R}}\left( \nu + \rho - \mu_{\beta}^{N^{\lambda}} \right) -\int_{\mathbf{R}^{d}} \log\left( \mu_{\beta}^{N^{\lambda}} \right) d\rho + \mbox{ent}[\rho] \Bigg),
\end{equation}
where the infimum is taken over $\rho$ such that
\begin{equation}
    \int_{\mathbf{R}^{d}} \nu + \rho - \mu_{\beta}^{N^{\lambda}}=0.
\end{equation}
The definition of $\mathbf{T}_{\lambda}^{N,\neq}$ is almost the same as $\mathbf{T}_{\lambda}^{N}$ but omitting the diagonal inside the square $\square_{R}$ in the computation of the Coulomb energy. This modification allows for the quantity to be finite for atomic measures inside the cube. 

Given $\alpha \in \mathbf{R}^{+}$ we also define
\begin{equation}\label{defofPhibeta}
    \mathbf{F}_{\square_{R}}^{\alpha}(\nu)=\inf_{\rho \in \mathcal{M}^{+}(\mathbf{R}^{d} \setminus \square_{R})} \mathcal{E}(\nu+\rho - \alpha),
\end{equation}
where the inf is taken over all $\rho \in C^{\infty}$ such that 
\begin{equation}\label{massconstrainteq}
    \int_{\mathbf{R}^{d}} \rho + \nu - \alpha \,  dx =  0. 
\end{equation}

We generalize the definition of $ \Phi_{\Omega}^{\alpha}(\nu)$ to a more general setting in which the background measure is not necessarily constant out of $\Omega$. Given a set $\Omega \in \mathbf{R}^{d}$ and a background measure $\mu \in \mathcal{M}(\mathbf{R}^{d}),$ 
\begin{equation}
    \Phi_{\Omega}^{\mu}(\nu)=\inf_{\rho \in \mathcal{M}^{+}(\mathbf{R}^{d} \setminus \Omega)} \mathcal{E}(\nu+\rho-\mu).
\end{equation}

We now define an analog of $\Phi^{\mu}_{\Omega}$ for measures that are not absolutely continuous. Given a measurable set $\Omega \subset \mathbf{R}^{d}$, a positive measure $\mu$ on $\mathbf{R}^{d}$, $ \Phi_{\Omega, \neq}^{\mu}(\nu)$ is defined for a measure $\nu$ on $\Omega$ as
\begin{equation}
    \Phi_{\Omega, \neq}^{\mu}(\nu)=\inf_{\rho \in \mathcal{M}^{+}(\mathbf{R}^{d} \setminus \Omega)} \mathcal{E}^{\neq}_{\Omega}(\nu+\rho-\mu).
\end{equation}

We also introduce the notation.
\begin{equation}
    \mathcal{G}(\mu, \nu) = \iint_{\mathbf{R}^{d} \times \mathbf{R}^{d}} g(x-y) d\mu({x}) d\nu({y}).
\end{equation}

\end{definition}

\begin{remark}
  Note that 
\begin{equation}
    ({\rm emp}_{N}'(Y_{N}))^{N^{\lambda}}={\rm lemp}_{N}.
\end{equation}

Note also that for any $\alpha \in \mathbf{R}^{+}$, $\mathcal{E}$ has the scaling relation
\begin{equation}
    \mathcal{E}(\mu^{\alpha})=\alpha^{d+2} \mathcal{E}(\mu),
\end{equation}
and therefore $ \mathbf{W}^{M, \mu}_{\Omega, \epsilon}(\nu)$ has the scaling relation
\begin{equation}
    \mathbf{W}^{M, \mu}_{\Omega, \epsilon}(\nu)=\alpha^{-(d+2)}  \mathbf{W}^{M, \mu^{\alpha}}_{\alpha\Omega, \alpha^{d} \epsilon}(\nu^{\alpha}).
\end{equation}

\end{remark}

Lastly, we introduce notation that will be used throughout the work. 

\begin{remark}[Notation]
\label{def:lamb}
Given $\epsilon > 0$, we denote by $\lambda_{\epsilon}$ the uniform probability measure on $\partial B(0,\epsilon)$.
\end{remark}

\section{Preliminary results}

In this section, we will prove some preliminary results needed for the main Theorem.  

We begin with a splitting formula around the thermal equilibrium measure, which is an analog of the usual splitting formula (see for example \cite{sandier20152d}).
\begin{proposition}\label{thermalsplittingformula}
 The Hamiltonian $\mathcal{H}_{N}$ can be split into:
 \begin{equation}
 \label{eq:splittingform}
         \mathcal{H}_{N}\left( X_{N} \right)=N^{2}\mathcal{E}_{\beta}\left( \mu_{\beta} \right)+N\sum_{i=1}^{N} \zeta_{\beta} \left( x_{i} \right) +\frac{N^{2}}{2} \mathcal{E}^{\neq}\left( {\rm emp}_{N}-\mu_{\beta} \right) ,
 \end{equation}
 where
 \begin{equation}
     \zeta_{\beta}=-\frac{1}{N\beta} \log\left( \mu_{\beta} \right). 
 \end{equation}
\end{proposition}

\begin{proof}
See \cite{armstrong2019local}.
\end{proof}

\begin{definition}
In analogy with previous work in this field \cite{armstrong2019local, leble2018fluctuations, bekerman2018clt, leble2017large1}, we define a next order partition function $K_{N,\beta},$ as 
\begin{equation}\label{definitionofnextorderpartitionfunction}
    K_{N,\beta}=Z_{N,\beta} \exp\left( -N^{2}\beta \mathcal{E}_{\beta} \left( \mu_{\beta} \right) \right).
\end{equation}
Using \eqref{definitionofnextorderpartitionfunction}, we may rewrite the Gibbs measure as
\begin{equation}
  d \mathbf{P}_{N, \beta}(x_{1}...x_{N}) = \frac{1}{K_{N, \beta}} \exp\left( -\frac{1}{2} N^{2}\beta \mathcal{E}^{\neq}({\rm emp}_{N}-\mu_{\beta} ) \right) \Pi_{i=1}^{N} \mu_{\beta}(x_{i}) d x_{i}.
\end{equation}
\end{definition}

We also need the following piece of information about $\mu_{\beta}$, which can be deduced from \cite{armstrong2019thermal}, Theorem 1.
\begin{remark}
\label{rem:unifconv}
Let $T>0, \lambda >0$ and assume that $\lim_{N \to \infty} N \beta = \infty$, and the potential $V$ is admissible, then
\begin{equation}
    \| \mu_{\beta}^{N^{\lambda}} - \mu_{V}(0)  \|_{L^{\infty} (B(0,T))} \to 0.
\end{equation}
\end{remark}

We proceed to prove some elementary properties about the rate functions in Theorem \ref{maintheorem}. 

\begin{claim}
For any $\alpha, R >0$, the function $\mathcal{N}(\nu| \alpha \mathbf{1}_{\square_{R}} )$ is a convex (in $\nu$) rate function. 
\end{claim}

\begin{proof}
Since convexity and l.s.c. are immediate from the convexity and l.s.c. of ent, we need only show that $\mathcal{N}(\nu| \alpha \mathbf{1}_{\square_{R}} )$ is positive for any $\nu \in \mathcal{M}^{+}(\square_{R})$. Throughout the proof, we will use the notation 
\begin{equation}
    \overline{\nu} = \frac{1}{R^{d}} \int_{\square_{R}} \nu dx.
\end{equation}
Using Jensen's inequality, the convexity of $x \mapsto x\log(x),$ and doing a first-order Taylor expansion of $x\log(x),$ we have
\begin{equation}
    \begin{split}
        \mathcal{N}(\nu| \alpha \mathbf{1}_{\square_{R}} ) &= \int_{\square_{R}} \log\left( \frac{\nu}{ \alpha} \right) \nu dx + R^{d} \alpha - |\nu|\\
        &= \int_{\square_{R}} \log\left( \frac{\nu}{ \alpha} \right) \frac{\nu}{ \alpha }  \alpha dx + R^{d} \alpha - |\nu|\\
        &\geq \int_{\square_{R}} \log\left( \frac{\overline{\nu}}{ \alpha} \right) \frac{\overline{\nu}}{ \alpha }  \alpha dx + R^{d} \alpha - |\nu|\\
        &= R^{d}\alpha \log\left( \frac{\overline{\nu}}{ \alpha} \right) \frac{\overline{\nu}}{ \alpha }   + R^{d} \alpha - |\nu|\\
        &\geq R^{d}\alpha \left(  \frac{\overline{\nu}}{ \alpha} -1 \right)  + R^{d} \alpha - |\nu|\\
        &= 0.
    \end{split}
\end{equation}
\end{proof}

The following claim is standard and can be found, for example, in \cite{padilla2023concentration}.

\begin{lemma}
\label{lem:fund}
The energy $\mathcal{E}$ is l.s.c. w.r.t. to weak $H^{-1}$ convergence.
\end{lemma}

With the help of Lemma \ref{lem:fund}, we can prove some elementary properties about $\Phi_{\square_{R}}^{\alpha}$. 

\begin{lemma}
For any $R, \alpha >0$ and any measure $\mu$ in $\square_{R}$ such that $\mathcal{E}(\mu) < \infty$, the infimum in the definition of $\Phi_{\square_{R}}^{\alpha} (\mu)$ is achieved.
\end{lemma}

\begin{proof}
Let $\rho_{N}$ be a minimizing sequence for 
\begin{equation}
    \inf_{\rho \in \mathcal{M}^{+}(\mathbf{R}^{d} \setminus \square_{R})}  \mathcal{E}(\mu + \rho - \alpha) .
\end{equation}
Note that
\begin{equation}
  \limsup_{N \to \infty}  \mathcal{E} (\mu+\rho_{N} - \alpha) < \infty.
\end{equation}
Hence, modulo a subsequence, 
\begin{equation}
    \rho_{N} \rightharpoonup \rho
\end{equation}
weakly in $H^{-1}$ for some $\rho \in \mathcal{M}^{+}(\mathbf{R}^{d} \setminus \square_{R})$. By l.s.c. of $\mathcal{E},$ we have
\begin{equation}
    \begin{split}
        \mathcal{E}(\mu+\rho - \alpha) & \leq \liminf_{N \to \infty}  \mathcal{E}(\mu+\rho_{N} - \alpha) \\
        &=   \inf_{\rho \in \mathcal{M}^{+}(\mathbf{R}^{d} \setminus \square_{R})}  \mathcal{E}(\mu + \rho - \alpha) .
    \end{split}
\end{equation}
\end{proof}

We now prove that the function $\Phi_{\square_{R}}^{\alpha}$ is a convex rate function for any $\alpha, R >0$. 

\begin{claim}
For any $\alpha, R >0$, the function $\Phi_{\square_{R}}^{\alpha}$ is a convex rate function. 
\end{claim}

\begin{proof}
We first prove convexity. Let $\mu, \nu$ be measures on $\square_{R}$ such that $\Phi_{\square_{R}}^{\alpha}(\mu) + \Phi_{\square_{R}}^{\alpha}(\nu) < \infty$.  Let 
\begin{equation}
    \rho_{\mu} = \operatorname{argmin}_{\rho \in \mathcal{M}^{+}(\mathbf{R}^{d} \setminus \square_{R})}  \mathcal{E}(\mu + \rho - \alpha) ,
\end{equation}
and 
\begin{equation}
    \rho_{\nu} = \operatorname{argmin}_{\rho \in \mathcal{M}^{+}(\mathbf{R}^{d} \setminus \square_{R})} \mathcal{E}(\nu + \rho - \alpha).
\end{equation}
Then, using the convexity of $\mathcal{E}$ we have
\begin{equation}
    \begin{split}
        \Phi_{\square_{R}}^{\alpha} \left(\frac{1}{2} \left(\mu+\nu\right) \right) &\leq \mathcal{E} \left( \frac{1}{2} (\mu+\nu)  + \frac{1}{2} (\rho_{\mu}+\rho_{\nu}) - \alpha)  \right) \\
        &\leq  \frac{1}{2} \bigg( \mathcal{E} \left(  \mu  +  \rho_{\mu} - \alpha  \right) + \mathcal{E} \left(  \mu  +  \rho_{\mu} - \alpha  \right)\bigg) \\
        &= \frac{1}{2} \left(  \Phi_{\square_{R}}^{\alpha} ( \mu ) +  \Phi_{\square_{R}}^{\alpha} ( \nu ) \right).
    \end{split}
\end{equation}
This proves the convexity of $\Phi_{\square_{R}}^{\alpha}$. We now turn to prove that  $\Phi_{\square_{R}}^{\alpha} $ is l.s.c. Since it is clearly positive, this will conclude the proof. Let $\mu$ be a measure in $\square_{R}$ such that $\Phi_{\square_{R}}^{\alpha}(\mu)<\infty$ and let $\mu_{n}$ be a sequence of measures in $\square_{R}$ such that
\begin{equation}
    \mu_{n} \rightharpoonup \mu
\end{equation}
weakly in the sense of measures. Let 
\begin{equation}
    \rho_{n} = \operatorname{argmin}_{\rho \in \mathcal{M}^{+}(\mathbf{R}^{d} \setminus \square_{R})}  \left( \mathcal{E}(\mu_{n} + \rho - \alpha) \right).
\end{equation}
Note that
\begin{equation}
   \limsup_{n \to \infty} \mathcal{E}(\mu_{n} + \rho_{n} - \alpha) < \infty.
\end{equation}
Then by precompactness, we have that the sequence $\mu_{n}+\rho_{n} - \alpha $ is precompact in the weak $H^{-1}$ topology (note that we are not claiming precompactness for convergence in the BL metric, which is clearly not true in general). Let $\sigma$ be such that
\begin{equation}
    \mu_{n} + \rho_{n}  - \alpha  \rightharpoonup \sigma.
\end{equation}
It is easy to see that $\sigma$ and $\mu- \alpha \mathbf{1}_{\square_{R}}$ agree in the interior of $\square_{R}.$ Note also that
\begin{equation}
    \rho := \sigma - (\mu- \alpha \mathbf{1}_{\square_{R}})
\end{equation}
is a positive measure, and therefore it can be used as a test function in the definition of $\Phi_{\square_{R}}^{\alpha} $. Then, using l.s.c. of $\mathcal{E}$ we have
\begin{equation}
    \begin{split}
         \Phi_{\square_{R}}^{\alpha} (\mu) &\leq \mathcal{E} \left( \mu + \rho - \alpha) \right) \\
         &\leq \liminf_{n \to \infty} \mathcal{E} \left( \mu_{n} + \rho_{n} - \alpha) \right) \\
         &= \liminf_{n \to \infty} \Phi_{\square_{R}}^{\alpha} (\mu_{n}). 
    \end{split}
\end{equation}
\end{proof}

We will now prove that the rate functions are good.
\begin{claim}
For any $R, \alpha>0,$ the function $\mathcal{N}(\mu|\alpha\mathbf{1}_{\square_{R}})$ is a good rate function, i.e. sublevel sets are precompact in the topology of weak convergence of measures.
\end{claim}

\begin{proof}
Consider the sublevel sets
\begin{equation}
    L_{M} = \{ \mu \in \mathcal{M}^{+}(\square_{R}) | \mathcal{N}(\mu|\alpha\mathbf{1}_{\square_{R}}) < M \}.
\end{equation}
We will prove that there exists $N$ such that if 
\begin{equation}
    \mu \in L_{M}
\end{equation}
then
\begin{equation}
    |\mu| \leq N,
\end{equation}
which will imply the desired compactness. Let 
\begin{equation}
    \overline{\mu} = \frac{1}{R^{d}} \int_{\square_{R}} d \mu.
\end{equation}
Using Jensen's inequality, we have
\begin{equation}
    \begin{split}
        \mathcal{N}(\mu|\alpha\mathbf{1}_{\square_{R}}) &\geq  \mathcal{N}(\overline{\mu}| \alpha \mathbf{1}_{\square_{R}})\\
        &= R^{d} \alpha \left( \frac{\overline{\mu}}{\alpha} \log \left( \frac{\overline{\mu}}{\alpha} \right) -  \frac{\overline{\mu}}{\alpha} +1 \right).
    \end{split}
\end{equation}
Since $x \log(x) - x \to \infty$ as $x \to \infty,$ we have that there exists $N$ such that $|\mu| < N$ if $\mu \in L_{M}.$ Hence, $L_{M}$ is precompact in the topology of weak convergence.
\end{proof}

We now prove that $\Phi^{\alpha}_{\square_{R}}$ is a good rate function.

\begin{claim}
For any $R, \alpha >0$ the function $\Phi^{\alpha}_{\square_{R}}$ is a good rate function, i.e. sublevel sets are precompact in the topology of weak convergence of measures..
\end{claim}

\begin{proof}
   Let $\mu_{n} \in \mathcal{M}^{+}(\square_{R})$ be such that 
\begin{equation}
\label{eq:finite}
    \limsup_{n \to \infty } \Phi^{\alpha}_{\square_{R}}(\mu_{n}) < \infty. 
\end{equation} 

Let 
\begin{equation}
    \rho_{n} = \operatorname{argmin}_{\rho \in \mathcal{M}^{+}(\mathbf{R}^{d} \setminus \square_{R})}   \mathcal{E}(\mu_{n} + \rho - \alpha).
\end{equation}

Since we are assuming equation \eqref{eq:finite}, we have that $\mu_{n} + \rho_{n}$ converges, modulo a subsequence (not relabelled) weakly in the $H^{-1}$ topology. Hence the restriction to $\square_{R}$, $\mu_{n}$ converges weakly in the $H^{-1}$ topology. In particular, 
\begin{equation}
\label{eq:finite2}
    \limsup_{n \to \infty} \mathcal{E}(\mu_{n}) < \infty.
\end{equation}

Since $\mu_{n}$ is a positive measure, equation \eqref{eq:finite2} implies that
\begin{equation}
    \limsup_{n \to \infty} |\mu_{n}|< \infty,
\end{equation}
which implies that modulo a subsequence (not relabelled) $\mu_{n}$ converges in the topology of weak convergence of probability measures. 

\end{proof}

\section{Proof of upper bound}

In this section, we prove the upper bound of Theorem \ref{maintheorem}. Recall that we use the notation
\begin{equation}
    X_{N} = (y_{1}, ...y_{i_{N}}, z_{1}, ... z_{j_{N}}),
\end{equation}
with 
\begin{equation}
    y_{m} \in \square_{R}, \quad z_{m} \notin \square_{R}.
\end{equation}

\begin{proof}[Proof of Theorem \ref{maintheorem}, upper bound]

We begin by using the splitting formula for the thermal equilibrium measure (Proposition \eqref{thermalsplittingformula}). Let $\epsilon, R >0$ and $\nu \in \mathcal{M}^{+}(\square_{R})$, then

\begin{equation}\label{lowerbound}
    \begin{split}
       & \mathbf{P}_{N, \beta} \left( {\rm lemp}_{N} \in B(\nu, \epsilon) \right) =\\
        & \frac{1}{Z_{N, \beta}} \int_{X_{N}: {\rm lemp}_{N} \in B(\nu, \epsilon)} \exp \left( -\beta \mathcal{H}(X_{N}) \right) dX_{N}=\\
        & \frac{1}{K_{N, \beta}} \int_{X_{N}: {\rm lemp}_{N} \in B(\nu, \epsilon)} \exp \left( -\frac{1}{2}\beta N^{2} \mathcal{E}^{\neq}({\rm emp}_{N}-\mu_{\beta}) \right) \Pi_{i=1}^{N} \mu_{\beta}(x_{i})dx_{i} \leq \\
        & \frac{1}{K_{N, \beta}} \int_{X_{N}: {\rm lemp}_{N} \in B(\nu, \epsilon)} \exp \left( -\frac{1}{2}\beta N^{2} \mathbf{W}_{\square_{\frac{R}{N^{\lambda}}}, \frac{1}{N}}^{N-i_{N}, \mu_{\beta}}({\rm emp}'_{N}) \right) \Pi_{i=1}^{N} \mu_{\beta}(x_{i})dx_{i} = \\
        &\frac{1}{K_{N, \beta}} \int_{X_{N}: {\rm lemp}_{N} \in B(\nu, \epsilon)}\\
        & \quad \exp \left( -\frac{1}{2}\beta N^{2-(d+2)\lambda} \mathbf{W}_{\square_{R}, N^{-1+\lambda d}}^{N -i_{N}, \mu_{\beta}^{N^{\lambda}}}({\rm lemp}_{N}) \right) \Pi_{i=1}^{N} \mu_{\beta}(x_{i})dx_{i} \leq \\
        &\frac{1}{K_{N, \beta}} \sup_{\mu \in B(\nu, \epsilon) \cap  \mathcal{A}_{i_{N}}^{N^{-1+\lambda d}}(\square_{R})} \\
        & \quad \left\{\exp \left( -\frac{1}{2}\beta N^{2-(d+2)\lambda} \mathbf{W}_{\square_{R}, N^{-1+\lambda d}}^{N-i_{N}, \mu_{\beta}^{N^{\lambda}}}(\mu) \right) \right\} \int_{X_{N}: {\rm lemp}_{N} \in B(\nu, \epsilon)} \Pi_{i=1}^{N} \mu_{\beta}(x_{i})dx_{i}.
    \end{split}
\end{equation}

In order to pass from the third to the fourth line, we have used that 
\begin{equation}
\begin{split}
    &\frac{1}{K_{N, \beta}} \int_{X_{N}: {\rm lemp}_{N} \in B(\nu, \epsilon)} \exp \left( -\frac{1}{2}\beta N^{2} \mathcal{E}^{\neq}({\rm emp}_{N}-\mu_{\beta}) \right) \Pi_{i=1}^{N} \mu_{\beta}(x_{i})dx_{i} \leq \\
     &\frac{1}{K_{N, \beta}} \int_{X_{N}: {\rm lemp}_{N} \in B(\nu, \epsilon)} \exp \left( -\frac{1}{2}\beta N^{2} \inf_{z_{i} \in \mathbf{R}^{d} \setminus \square_{R N^{-\lambda}}} \mathcal{E}^{\neq}({\rm emp}_{N}-\mu_{\beta}) \right) \Pi_{i=1}^{N} \mu_{\beta}(x_{i}) dx_{i},
\end{split}     
\end{equation}
since for any $Z_{N}^{*} \in ( \mathbf{R}^{d} \setminus \square_{R N^{-\lambda}} )^{j_{N}}$
\begin{equation}
     \mathcal{E}^{\neq}({\rm emp}_{N}(Y_{N}, Z_{N}^{*})-\mu_{\beta}) \geq  {\inf_{z_{i} \in \mathbf{R}^{d} \setminus \square_{R N^{-\lambda}}}} \mathcal{E}^{\neq}({\rm emp}_{N}(Y_{N}, Z_{N})-\mu_{\beta}).
\end{equation}

But given $y_{i} \in \square_{R N^{-\lambda}},$ we have
\begin{equation}
   \inf_{Z_{N} \in ( \mathbf{R}^{d} \setminus \square_{R N^{-\lambda}} )^{j_{N}}} \mathcal{E}^{\neq}({\rm emp}_{N}-\mu_{\beta}) {\geq} \mathbf{W}_{\square_{\frac{R}{N^{\lambda}}}, \frac{1}{N}}^{N-i_{N}, \mu_{\beta}}({\rm emp}'_{N}),
\end{equation}
see equation \eqref{definitiongeneralPhi}.

We now treat each of the terms in the last line of equation \eqref{lowerbound} individually. The second term is the easier, and we will will deal with it at the end of this section. More specifically, we will prove the following lemma:

\begin{lemma}\label{entropyterm}
Let $R, \epsilon >0$ and $\nu \in \mathcal{M}^{+}(\square_{R})$. Then
\begin{equation}
\begin{split}
     & \limsup_{N \to \infty} \Bigg( \frac{1}{N^{1-\lambda d}} \log \Bigg(  \int_{X_{N}: {\rm lemp}_{N} \in B(\nu, \epsilon)}  \Pi_{i=1}^{N} \mu_{\beta}(x_{i}) dx_{i} \Bigg) \Bigg) \leq \\
     & - \inf_{\mu \in B(\nu, \epsilon)} ( \mathcal{N}(\mu| \mu_{V}(0) \mathbf{1}_{\square_{R}}) ). 
\end{split}
\end{equation}

\end{lemma}

The analysis of the first term is more delicate, and we deal with it in section $6$. The result we prove is the following:
\begin{lemma}\label{passingtolimit}
Let $R, \epsilon >0$, let $\nu \in \mathcal{M}^{+}(\square_{R})$ and $i_{N}$ be an integer smaller than or equal to $N$. Then
\begin{equation}
         \inf_{\mu \in B(\nu, \delta)} \Phi_{\square_{R}}^{\mu_{V}(0)}(\mu) \leq \liminf_{N \to \infty} \inf_{\mu \in B(\nu, \delta)  \cap  \mathcal{A}_{i_{N}}^{N^{-1+\lambda d}}(\square_{R})} \mathbf{W}_{\square_{R}, N^{-1+\lambda d}}^{N-i_{N}, \mu_{\beta}^{N^{\lambda}}}(\mu).  
\end{equation}

Furthermore, for any $\nu \in \mathcal{M}^{+}(\square_{R})$ such that $\mathcal{E}(\nu) < \infty$ we have
\begin{equation}
 \lim_{N \to \infty} \left| \Phi^{\mu_{V}(0)}_{\square_{R}}(\nu) -  \Phi^{\mu_{\beta}^{N^{\lambda}}}_{\square_{R}}(\nu) \right| = 0. 
\end{equation}
\end{lemma}

We will now finish the proof of the upper bound in Theorem \ref{maintheorem} using Lemmas \ref{entropyterm} and \ref{passingtolimit}. We start with the last line of equation \eqref{lowerbound}: 
\begin{equation}
\begin{split}
        &\mathbf{P}_{N, \beta} \left( {\rm lemp}_{N} \in B(\nu, \epsilon) \right) \leq \\
        &\frac{1}{K_{N, \beta}} \sup_{\mu \in B(\nu, \epsilon)  \cap  \mathcal{A}_{i_{N}}^{N^{-1+\lambda d}}(\square_{R}) }\\
        & \quad \left\{\exp \left( -\frac{1}{2}\beta N^{2-(d+2)\lambda} \mathbf{W}_{\square_{R}, N^{-1+\lambda d}}^{N-i_{N}, \mu_{\beta}^{N^{\lambda}}}(\mu) \right) \right\} \int_{X_{N}: {\rm lemp}_{N} \in B(\nu, \epsilon)} \Pi_{i=1}^{N} \mu_{\beta}(x_{i})dx_{i}.
\end{split}        
\end{equation}

Using results from \cite{rougerie2016higher}, or from \cite{armstrong2019local}, we know that
\begin{equation}\label{nextorderpartitionfunctionissmall}
    |\log(K_{N, \beta}) | \leq C \beta N^{2-\frac{2}{d}},
\end{equation}
using the hypothesis that $\lambda < \frac{1}{d(d+2)}$ we have that
\begin{equation}
     |\log(K_{N, \beta}) | = o (\min ( \beta N^{2-\lambda(d+2)}, N^{1 - \lambda d}).
\end{equation}
Bounding this error term (and bounding a similar error term in the upper bound) is the only step in which we use the hypothesis that $\lambda < \frac{1}{d(d+2)}$.

Note that, if $\gamma < \gamma^{*}$ then 
\begin{equation}
    2-(d+2)\lambda-\gamma>1-\lambda d,
\end{equation}
and so
\begin{equation}
   \limsup_{N \to \infty} \frac{1}{\beta N^{2-(d+2)\lambda}}\log \left( \mathbf{P}_{N, \beta} \left( {\rm lemp}_{N} \in B(\nu, \epsilon) \right) \right) \leq -\frac{1}{2}\inf_{ \mu \in B(\nu, \epsilon)} \Phi^{\mu_{V}(0)}_{\square_{R}}(\mu).
\end{equation}

And finally, if $\gamma > \gamma^{*}$ then 
\begin{equation}
    2-(d+2)\lambda-\gamma<1-\lambda d,
\end{equation}
and so
\begin{equation}
   \limsup_{N \to \infty} \frac{1}{N^{1-\lambda d}}\log \left( \mathbf{P}_{N, \beta} \left( {\rm lemp}_{N} \in B(\nu, \epsilon) \right) \right) \leq  -\inf_{ \mu \in B(\nu, \epsilon)} \mathcal{N}[\mu|\mu_{V}(0)\mathbf{1}_{\square_{R}}].
\end{equation}
\end{proof}

This concludes the proof of the upper bound of Theorem \ref{maintheorem}. We now turn to the proof of the auxiliary lemmas (Lemmas \ref{entropyterm} and \ref{passingtolimit}). We start with Lemma \ref{entropyterm}, which we restate here for convenience:
\begin{lemma}\label{entropyterm2}
Let $R, \epsilon >0$ and $\nu \in \mathcal{M}^{+}(\square_{R})$. Then
\begin{equation}
\begin{split}
     & \limsup_{N \to \infty} \Bigg( \frac{1}{N^{1-\lambda d}} \log \Bigg(  \int_{X_{N}: {\rm lemp}_{N} \in B(\nu, \epsilon)}  \Pi_{i=1}^{N} \mu_{\beta}(x_{i}) dx_{i} \Bigg) \Bigg) \leq \\
     & - \inf_{\mu \in B(\nu, \epsilon)} ( \mathcal{N}(\mu| \mu_{V}(0) \mathbf{1}_{\square_{R}}) ). 
\end{split}
\end{equation}

\end{lemma}

\begin{proof}
Using Sanov's theorem and the scaling relation of ${\rm ent}$, we have that
\begin{equation}\label{startingpoint}
   \lim_{N \to \infty} \frac{1}{N^{1-\lambda d}} \log \left( \int_{X_{N}: {\rm lemp}_{N} \in B(\nu, \epsilon)}  \Pi_{i=1}^{N} \mu_{\beta}(x_{i}) dx_{i} \right) \leq - \inf_{\rho} {\rm ent}[{\rho} | \mu_{\beta}^{N^{\lambda}} ],
\end{equation}
where the infimum is taken over $\rho$ such that $|\rho|=N^{\lambda d}$ and $\rho|\square_{R} \in B(\nu, \epsilon).$ Note that we may rewrite equation \eqref{startingpoint} as
\begin{equation}
\label{eq:minprob}
\begin{split}
      &   \lim_{N \to \infty} \frac{1}{N^{1-\lambda d}} \log \left( \int_{X_{N}: {\rm lemp}_{N} \in B(\nu, \epsilon)}  \Pi_{i=1}^{N} \mu_{\beta}(x_{i}) dx_{i} \right) \leq  \\
      & -\inf_{ \mu \in B(\nu, \epsilon)} \left( {\rm ent}[ \mu | \mu_{\beta}^{N^{\lambda}}|_{
    \square_{R}}] + \inf_{\rho} {\rm ent}[\rho | \mu_{\beta} ] \right), 
\end{split}
\end{equation}
where the infimum is taken over all $\rho \in \mathcal{M}^{+} (\mathbf{R}^{d} \setminus \square_{R})$ such that $|\rho| = N^{\lambda d} - |\mu|.$

We first determine the optimal $\rho$ in the minimization problem of equation \eqref{eq:minprob} for a given $\mu$. This can be done by adding a Lagrange multiplier for the constraint of mass and then computing the Euler Lagrange equations. The solution is that the minimizer $\rho^{*}$ is given by
\begin{equation}
    \rho^{*} = \kappa \mu_{\beta}^{N^{\lambda}} \mathbf{1}_{\mathbf{R}^{d} \setminus \square_{R}},
\end{equation}
where $\kappa$ is given by
\begin{equation}
    \kappa = \frac{N^{\lambda d} - |\mu|}{ \int_{\mathbf{R}^{d} \setminus \square_{R}} \mu_{\beta}^{N^{\lambda}}}.
\end{equation}

Hence we have that, for each $\mu \in B(\nu, \epsilon),$
\begin{equation}
    \begin{split}
        \lim_{N \to \infty}{\rm ent}[\mu + \rho^{*}| \mu_{\beta}^{N^{\lambda}}] &=  \lim_{N \to \infty} {\rm ent}[\mu | \mu_{\beta}^{N^{\lambda}}|_{\square_{R}}]+ \int_{\mathbf{R}^{d} \setminus \square_{R}} \log(\kappa) \kappa \mu_{\beta}^{N^{\lambda}} dx\\
        &= {\rm ent}[\mu | \mu_{V}(0) \mathbf{1}_{\square_{R}}] +  \lim_{N \to \infty} \kappa(\kappa-1) \int_{\mathbf{R}^{d} \setminus \square_{R}} \mu_{\beta}^{N^{\lambda}} dx .
    \end{split}
\end{equation}

In the last equation, we have used Remark \ref{rem:unifconv} and the approximation $\log(\kappa) \simeq \kappa-1$, since $\kappa$ tends to $1$ as $N$ tends to $\infty$. Recalling that \begin{equation}
    \int_{\mathbf{R}^{d}} \mu_{\beta}^{N^{\lambda}} = N^{\lambda d},
\end{equation} 
and using again Remark \ref{rem:unifconv} we have that
\begin{equation}
    \lim_{N \to \infty} \kappa(\kappa-1) \int_{\mathbf{R}^{d} \setminus \square_{R}} \mu_{\beta}^{N^{\lambda}} dx = R^{d} \mu_{V}(0) - |\mu|.
\end{equation}
Therefore
\begin{equation}
\begin{split}
     & \limsup_{N \to \infty} \Bigg( \frac{1}{N^{1-\lambda d}} \log \Bigg(  \int_{X_{N}: {\rm lemp}_{N} \in B(\nu, \epsilon)}  \Pi_{i=1}^{N} \mu_{\beta}(x_{i}) dx_{i} \Bigg) \Bigg) \leq \\
     & - \inf_{\mu \in B(\nu, \epsilon)} ( \mathcal{N}(\mu| \mu_{V}(0) \mathbf{1}_{\square_{R}}) ). 
\end{split}
\end{equation}
\end{proof}

\section{Proof of Lemma \ref{passingtolimit}}

In this section, we prove Lemma \ref{passingtolimit}, which we restate here for convenience:
\begin{lemma}\label{passingtolimit2}
Let $R, \epsilon >0$, let $\nu \in \mathcal{M}^{+}(\square_{R})$ and $i_{N}$ be an integer smaller than or equal to $N$. Then
\begin{equation}
         \inf_{\mu \in B(\nu, \delta)} \Phi_{\square_{R}}^{\mu_{V}(0)}(\mu) \leq \liminf_{N \to \infty} \inf_{\mu \in B(\nu, \delta)  \cap  \mathcal{A}_{i_{N}}^{N^{-1+\lambda d}}(\square_{R}) } \mathbf{W}_{\square_{R}, N^{-1+\lambda d}}^{N-i_{N}, \mu_{\beta}^{N^{\lambda}}}(\mu).  
\end{equation}

Furthermore, for any $\nu \in \mathcal{M}^{+}(\square_{R})$ such that $\mathcal{E}(\nu) < \infty$ we have
\begin{equation}
 \lim_{N \to \infty} \left| \Phi^{\mu_{V}(0)}_{\square_{R}}(\nu) -  \Phi^{\mu_{\beta}^{N^{\lambda}}}_{\square_{R}}(\nu) \right| = 0. 
\end{equation}
\end{lemma}

The idea is that, on the one hand, given our choice of dilation, $({\rm emp}_{N})^{N^{\lambda}}$ will converge to a continuous measure on every compact set. This implies that we can replace the infimum over purely atomic measures with the infimum over absolutely continuous measures in $\mathcal{M}^{+}(\mathbf{R}^{d} \setminus \square_{R}).$ On the other hand, $\mu_{\beta}^{N^{\lambda}}$ will converge to $\mu_{V}(0)$ on compact sets, so we can replace the background measure $\mu_{\beta}^{N^{\lambda}}$ with $\mu_{V}(0).$ We will now make this intuition more rigorous. 

\begin{proof}

\textbf{\textit{Step 1}}

We claim that
\begin{equation}
     \liminf_{N \to \infty} \inf_{\mu \in B(\nu, \delta) }  \Phi_{\square_{R}}^{\mu_{\beta}^{N^{\lambda}}}(\mu)  \leq  \liminf_{N \to \infty} \inf_{\mu \in B(\nu, \delta)  \cap  \mathcal{A}_{i_{N}}^{N^{-1+\lambda d}}(\square_{R}) }  \mathbf{W}_{\square_{R },N^{\lambda d -1}}^{N-i_{N}, \mu_{\beta}^{N^{\lambda}}}(\mu) .
\end{equation}

To see this, let 
\begin{equation}
    \mu = \frac{1}{N} \sum_{i=1}^{i_{N}} \delta_{y_{i}},
\end{equation}
and
\begin{equation}
  \rho= \frac{1}{N}\sum_{i=1}^{j_{N}} \delta_{z_{i}},
\end{equation}
with $y_{i} \in \square_{R}$, $z_{i} \in \mathbf{R}^{d} \setminus \square_{R}$ and $i_{N} + j_{N} = N$. 

Now we define 
\begin{equation}
    \widetilde{\mu}=\mu \ast \lambda_{N^{-\frac{1}{d}}}
\end{equation}
and
\begin{equation}
    \widetilde{\rho} = \rho \ast \lambda_{N^{-\frac{1}{d}}},
\end{equation}
(see Remark \ref{def:lamb} for notation).

Then 
\begin{equation}
    \| \mu - \widetilde{\mu} \|_{BL} \leq N^{-\frac{1}{d}}
\end{equation}
and we also have, because of Lemmas \ref{smearinglemma1}, \ref{smearinglemma2}, \ref{smearinglemma3}, \ref{smearinglemma4} that
\begin{equation}
     \mathcal{E}(\widetilde{\mu}+\widetilde{\rho}-\mu_{\beta}) \leq  \mathcal{E}^{\neq}(\mu+{\rho}-\mu_{\beta})  + C N^{-\frac{2}{d}}.
\end{equation}

Note that $C$ depends only on $V$ and $d$, since $\mu_{\beta}$ is uniformly bounded in $N$ for $N$ large enough, with a bound that depends only on $V$ and $d$. 

Using the hypothesis that 
\begin{equation}
    \lambda < \frac{1}{d(d+2)},
\end{equation}
we have that
\begin{equation}
    N^{-\frac{2}{d}} << N^{- \lambda(d+2)},
\end{equation}
which implies, using the scaling relations of $\Phi_{\square_{R}}^{\mu_{\beta}^{N^{\lambda}}}$ and $\mathbf{W}_{\square_{R },N^{\lambda d -1}}^{N-i_{N}, \mu_{\beta}^{N^{\lambda}}}$, that
\begin{equation}\label{almostthere}
     \liminf_{N \to \infty} \inf_{\mu \in B(\nu, \delta)}  \Phi_{\square_{R}}^{\mu_{\beta}^{N^{\lambda}}}(\mu)  \leq  \liminf_{N \to \infty} \inf_{\mu \in B(\nu, \delta)  \cap  \mathcal{A}_{i_{N}}^{N^{-1+\lambda d}}(\square_{R}) }  \mathbf{W}_{\square_{R },N^{\lambda d -1}}^{N-i_{N}, \mu_{\beta}^{N^{\lambda}}}(\mu).
\end{equation}

\textbf{\textit{Step 2}} We now prove the second part of the claim: that for any $\nu \in \mathcal{M}^{+}(\square_{R})$ such that $\mathcal{E}(\nu) < \infty$ we have
\begin{equation}
 \lim_{N \to \infty} \left| \Phi^{\mu_{V}(0)}_{\square_{R}}(\nu) -  \Phi^{\mu_{\beta}^{N^{\lambda}}}_{\square_{R}}(\nu) \right| = 0. 
\end{equation}

We will first prove that
\begin{equation}
      \limsup_{N \to \infty}  \Phi^{\mu_{\beta}^{N^{\lambda}}}_{\square_{R}}(\nu) \leq \Phi^{\mu_{V}(0)}_{\square_{R}}(\nu) .
\end{equation}
To this end, let
\begin{equation}
    \overline{\rho} = \operatorname{argmin}_{\rho \in \mathcal{M}^{+}(\mathbf{R}^{d} \setminus \square_{R})} \mathcal{E} (\rho + \nu - \mu_{V}(0)).
\end{equation}

For any $\epsilon > 0$ let $T$ be such that
\begin{equation}
    \begin{split}
        &\left| \mathcal{E} \left( \left(\overline{\rho} + \nu - \mu_{V}(0)\right)\mathbf{1}_{B(0,T)} \right) - \mathcal{E} \left( \overline{\rho} + \nu - \mu_{V}(0) \right) \right| \leq \epsilon.
    \end{split}
\end{equation}

Taking $ \overline{\rho} \mathbf{1}_{B(0,T)} + \mu_{\beta}^{N^{\lambda}} \mathbf{1}_{\mathbf{R}^{d} \setminus B(0,T)} $ as a test function in the definition of $\Phi^{\mu_{\beta}^{N^{\lambda}}}_{\square_{R}}(\nu)$ and using Remark \ref{rem:unifconv} we have
\begin{equation}
    \begin{split}
       \limsup_{N \to \infty}  \Phi^{\mu_{\beta}^{N^{\lambda}}}_{\square_{R}}(\nu) &\leq  \limsup_{N \to \infty} \mathcal{E} \left( \left(\overline{\rho} + \nu - \mu_{\beta}^{N^{\lambda}} \right)\mathbf{1}_{B(0,T)} \right) \\
      &=  \mathcal{E} \left( \left(\overline{\rho} + \nu - \mu_{V}(0)\right)\mathbf{1}_{B(0,T)} \right) \\
      & \leq \mathcal{E} (\overline{\rho} + {\nu} - \mu_{V}(0)) +\epsilon  \\
      &= \Phi^{\mu_{V}(0)}_{\square_{R}}(\nu) + \epsilon.
    \end{split}
\end{equation}

Since $\epsilon>0$ is arbitrary, we can conclude that 
\begin{equation}
      \limsup_{N \to \infty}  \Phi^{\mu_{\beta}^{N^{\lambda}}}_{\square_{R}}(\nu) \leq \Phi^{\mu_{V}(0)}_{\square_{R}}(\nu) .
\end{equation}

We now turn to prove
\begin{equation}
        \Phi^{\mu_{V}(0)}_{\square_{R}}(\nu) \leq  \liminf_{N \to \infty} \Phi^{\mu_{\beta}^{N^{\lambda}}}_{\square_{R}}(\nu). 
\end{equation}
To this end, let 
\begin{equation}
    {\rho}_{N} = \operatorname{argmin}_{\rho } \mathcal{E} (\rho + \nu - \mathbf{1}_{\square_{R}} \mu_{\beta}^{N^{\lambda}}),
\end{equation}
where $\rho$ is minimized over measures satisfying $\rho \geq -\mu_{\beta}^{N^{\lambda}}$ and which are supported in $\mathbf{R}^{d} \setminus \square_{R}.$

Note that 
\begin{equation}
    \Phi^{\mu_{\beta}^{N^{\lambda}}}_{\square_{R}} (\nu) = \mathcal{E} (\rho_{N} + \nu - \mathbf{1}_{\square_{R}} \mu_{\beta}^{N^{\lambda}}).
\end{equation}

Then, since 
\begin{equation}
    \limsup_{N \to \infty} \mathcal{E} (\rho_{N} + \nu - \mathbf{1}_{\square_{R}} \mu_{\beta}^{N^{\lambda}}) < \infty,
\end{equation}
we have that
\begin{equation}
    \rho_{N} \rightharpoonup \widehat{\rho},
\end{equation}
weakly in $H^{-1},$ for some $\widehat{\rho}$. It is easy to check that $\widehat{\rho} \geq -\mu_{V}(0)$ a.e. Using l.s.c. of $\mathcal{E},$ we then have that

\begin{equation}
    \begin{split}
        \Phi^{\mu_{V}(0)}_{\square_{R}}(\nu) &\leq \mathcal{E} (\nu -\mu_{V}(0)\mathbf{1}_{\square_{R}} + \widehat{\rho}) \\
        &\leq \liminf_{N \to \infty} \mathcal{E} (\nu -\mu_{\beta}^{N^{\lambda}}\mathbf{1}_{\square_{R}} + {\rho}_{N})\\
        &= \liminf_{N \to \infty} \Phi^{\mu_{\beta}^{N^{\lambda}}}_{\square_{R}}(\nu). 
    \end{split}
\end{equation}

\textbf{\textit{Step 3}}  

We now prove the first part of the statement of Lemma \ref{passingtolimit}. In view of equation \eqref{almostthere}, we will prove that for any $\delta>0$ and any measure $\nu$ on $\square_{R}$ such that $\mathcal{E}(\nu) < \infty$,
\begin{equation}
         \inf_{\mu \in B(\nu, \delta)} \Phi_{\square_{R}}^{\mu_{V}(0)}(\mu) \leq \liminf_{N \to \infty} \inf_{\mu \in B(\nu, \delta)} \Phi^{\mu_{\beta}^{N^{\lambda}}}_{\square_{R}}(\mu).  
\end{equation}

To this end, let $\mu_{N} \in  B(\nu, \delta)$ be such that 
\begin{equation}
    \inf_{\mu \in B(\nu, \delta)} \Phi^{\mu_{\beta}^{N^{\lambda}}}_{\square_{R}}(\mu) =  \Phi^{\mu_{\beta}^{N^{\lambda}}}_{\square_{R}}(\mu_{N}), 
\end{equation}
we assume that the infimum is achieved for clarity of exposition, otherwise, we could prove the claim up to an arbitrary error by taking a minimizing sequence. 

Since $\mu_{N} \in B(\nu, \delta),$ we have that as $N$ tends to $\infty$
\begin{equation}
   \mu_{N} \rightharpoonup \overline{\mu},
\end{equation}
weakly in the sense of measures, for some $\overline{\mu} \in B(\nu, \delta).$ Let 
\begin{equation}
    {\rho}_{N} = \operatorname{argmin}_{\rho} \mathcal{E} (\mu + \mu_{N} - \mathbf{1}_{\square_{R}} \mu_{\beta}^{N^{\lambda}}),
\end{equation}
where $\rho$ is minimized over $\rho \geq -\mu_{\beta}^{N^{\lambda}}$ supported in $\mathbf{R}^{d} \setminus \square_{R}.$

Note that
\begin{equation}
    \limsup_{N \to \infty} \mathcal{E} (\mu_{N} + \rho_{N} - \mathbf{1}_{\square_{R}} \mu_{\beta}^{N^{\lambda}}) < \infty,
\end{equation}
therefore, for a subsequence
\begin{equation}
    \rho_{N} + \mu_{N} \rightharpoonup \overline{\rho} + \overline{\mu},
\end{equation}
weakly in $H^{-1},$ for some $\overline{\rho} \geq -\mu_{V}(0)$. Therefore we can use $\overline{\rho}$ as a test function in the definition of $\Phi_{\square_{R}}^{\mu_{V}(0)}$ and get
\begin{equation}
    \begin{split}
         \inf_{\mu \in B(\nu, \delta)} \Phi_{\square_{R}}^{\mu_{V}(0)}(\mu) &\leq  \Phi_{\square_{R}}^{\mu_{V}(0)}(\overline{\mu}) \\
         &\leq \mathcal{E} (\overline{\rho} + \overline{\mu} - \mathbf{1}_{\square_{R}} \mu_{V}(0)) \\
         &\leq \liminf_{N \to \infty} \mathcal{E} ({\mu}_{N} + \rho_{N} - \mathbf{1}_{\square_{R}} \mu_{\beta}^{N^{\lambda}}) \\
         & = \liminf_{N \to \infty} \inf_{\mu\in B(\nu, \delta)} \Phi^{\mu_{\beta}^{N^{\lambda}}}_{\square_{R}}(\mu).
    \end{split}
\end{equation}

\end{proof}

\section{Proof of lower bound}

This section is devoted to proving the lower bound of the LDP's of Theorem \ref{maintheorem}. The approach will be to construct a family of point configurations that has correct energy and sufficient volume. 

We start with a lemma, which builds upon a construction found in unpublished class notes by Sylvia Serfaty. 

\begin{lemma}\label{goodenergygoodvolume}
Let $\mu_{n}, \overline{\nu}$ be probability measures on a compact set $\Omega$ such that 
\begin{equation}
   \limsup_{n \to \infty} {\rm ent}[\mu_{n}]< \infty, \quad  {\rm ent}[\overline{\nu}]< \infty,
\end{equation}
$\overline{\nu} \in L^{\infty}(\Omega)$, and 
\begin{equation}
    \mathcal{E}(\overline{\nu}) < \infty.
\end{equation}
Assume that $\mu_{n}(x)$ is uniformly equi-continuous and bounded away from $0$ uniformly in $x$ and $n$.
Then for every $\epsilon, \delta, \eta, $ there exists a family of configurations
\begin{equation}
    \Lambda_{\delta}^{\eta, \epsilon} \subset \mathbf{R}^{d \times n}
\end{equation}
such that
\begin{itemize}
    \item[$\bullet$] \begin{equation}
    {\rm emp}_{n}(X_{n}) \in B(\overline{\nu}, \epsilon)
\end{equation}
for any $X_{n} \in  \Lambda_{\delta}^{\eta, \epsilon}.$
 \item[$\bullet$] 
 \begin{equation}
     \liminf_{n \to \infty} \frac{1}{n} \log \left( \int_{X_{n} \in  \Lambda_{\delta}^{\eta, \epsilon}} \Pi_{i=1}^{n} \mu_{n}(x_{i}) dX_{n} \right) \geq - \liminf_{n \to \infty}{\rm ent}[\overline{\nu}|\mu_{n}] - \delta. 
 \end{equation}
  \item[$\bullet$] 
  \begin{equation}\label{goodenergy}
     \limsup_{n \to \infty} \left| \mathcal{E}^{\neq} ({\rm emp}_{n}(X_{n}) - \overline{\nu} )  \right| \leq \eta^{2}. 
  \end{equation}
  \item There exists $r>0$ such that
  \begin{equation}
      d(x_{i}, \partial \Omega) > r n^{-\frac{1}{d}} \quad {\rm and} \quad d(x_{i}, x_{j})> r n^{-\frac{1}{d}},
  \end{equation}
  for $i \neq j$.
\end{itemize}

\end{lemma}

The proof of Lemma \ref{goodenergygoodvolume} is found in Section \ref{App:construction}. 

We will also require the following lemma, which deals with approximating certain partition functions. 

\begin{lemma}
\label{boundonZYN}
 Let $Y_{N}=(y_{1}, y_{2}, ... y_{i_{N}})$ with $y_{j} \in \square_{R N^{-\lambda}}$ for each $j$.
Let $\nu \in \mathcal{M}^{+}(\square_{R})$ such that
\begin{equation}
    \mathcal{E}(\nu) < \infty, \quad \nu \in L^{\infty}(\square_{R}).
\end{equation}
Assume that
 \begin{equation}\label{muclosetoempN}
  \limsup_{N \to \infty}  \left|  \mathcal{E}^{\neq} ({\rm lemp}_{N}(Y_{N}) - \nu )  \right| \leq \eta^{2},
 \end{equation}
also that
\begin{equation}
\label{eq:massmatching}
    \lim_{N \to \infty} |\nu| - N^{-1+\lambda d} i_{N} \to 0,
\end{equation}
and that there exists $r>0$ such that
\begin{equation}
    d(y_{i}, \partial \square_{R N^{-\lambda}}) \geq r N^{-\frac{1}{d}}
\end{equation}
and
\begin{equation}
    d(y_{i}, y_{j}) \geq r N^{-\frac{1}{d}}.
\end{equation}
 Let
\begin{equation}
\label{eq:modZ}
    Z_{N, \beta}^{Y_{N}}=\iint_{Z_{N} \in (\mathbf{R}^{d} \setminus \square_{R N^{-\lambda}})^{j_{k}}} \exp \left( - \beta \mathcal{H}_{N} (Y_{N}, Z_{N}) \right) dZ_{N}.
\end{equation}

Then for $\gamma < \gamma^{*}$ we have 

\begin{equation}
 \frac{1}{\beta N^{2 - \lambda(d+2)}} \left(  -\log ( Z_{N, \beta}^{Y_{N}}) - N^{2} \beta \mathcal{E}_{\beta}(\mu_{\beta}) \right) \leq   \frac{1}{2}\Phi^{\mu_{V}(0)}_{\square_{R}}(\nu)+ C \eta  + o_{N}(1).
\end{equation}
where $C$ and $o_{N}(1)$ are independent of $Y_{N}.$

For $\gamma > \gamma^{*}$ we have 

\begin{equation}
\begin{split}
 &N^{1+\lambda d} \beta \left(  -\frac{1}{ N^{2} \beta}\log ( Z_{N, \beta}^{Y_{N}}) - \mathcal{E}_{\beta}(\mu_{\beta}) \right) \leq  \\
 &\int_{\mathbf{R}^{d}} \log(\mu_{\beta}^{N^{\lambda}}) d( {\rm lemp}_{N}(Y_{N})) -|\nu|+R^{d} \mu_{V}(0) + o_{N}(1),
\end{split} 
\end{equation}
where $o_{N}(1)$ is independent of $Y_{N}.$
\end{lemma}

\begin{proof}

 We will divide the proof in 3 steps. The idea of the proof is that using the variational formulation of the partition function, as well as the splitting formula for the equilibrium measure (Proposition \eqref{thermalsplittingformula}), we can reduce the integral in equation \eqref{eq:modZ} to 
\begin{equation}
\begin{split}
    &-\log ( Z_{N, \beta}^{Y_{N}}) \simeq N^{2} \beta \bigg( \mathcal{E}_{\beta}(\mu_{\beta})+\inf_{\rho} \frac{1}{2}\mathcal{E}^{\neq}_{\square_{R N^{-\lambda}}}({\rm emp}_{N}'(Y_{N})+\rho - \mu_{\beta}) -\\
    & \frac{1}{N \beta}\int_{\mathbf{R}^{d}} \log(\mu_{\beta}) d( \rho + {\rm emp}_{N}'(Y_{N}))+  \frac{1}{N \beta} {\rm ent}[\rho] \bigg).
\end{split}
\end{equation}
This is done in step 1. Steps 2 and 3 simplify this expression, and show that either the electric energy or the entropy dominates, depending on whether $\gamma > \gamma^{*}$ or $\gamma<\gamma^{*}.$

{\textbf{\textit{Step 1}}

We start with the characterization
\begin{equation}
\label{eq:genchar}
       -\frac{\log (Z_{N, \beta}^{Y_{N}})}{ \beta } =  \min_{\mu \in \mathcal{P}([\mathbf{R}^{d} \setminus \square_{R N^{-\frac{1}{d}}}]^{N})} \mathcal{F}(\mu), 
\end{equation}
in this equation,
\begin{equation}
   \mathcal{F}(\mu) = \int_{\mathbf{R}^{N d}} \mu(Z_{N})  \mathcal{H}^{Y_{N}}_{N}(Z_{N}) \, \mathrm d Z_{N} + \frac{1}{\beta} \int_{\mathbf{R}^{d N}} \mu(Z_{N}) \log(\mu(Z_{N})) \ \mathrm d Z_{N}
\end{equation}
where
\begin{equation}
    \mathcal{H}^{Y_{N}}_{N}\left( Z_{N}\right)= \frac{1}{2}  \sum_{y_{i}, y_{j} \in \square_{R N^{-\frac{1}{d}}}}g\left( y_{i}-y_{j} \right)+\frac{1}{2} \sum_{z_{i}, z_{j} \notin \square_{R N^{-\frac{1}{d}}}}g\left( z_{i}-z_{j} \right) +N \sum_{i} \widetilde{V}\left( x_{i} \right),
\end{equation}
and
\begin{equation}
     \widetilde{V} =V+ \frac{1}{2} {\rm emp}_{N}'(Y_{N}) \ast g. 
\end{equation}
This is a particular case of a characterization of the partition function which is valid in general, see for example \cite{rougerie2016higher}. 

We now define 
\begin{equation}
    \mu_{\beta}^{Y_{N}} = \text{argmin}_{\rho \in \mathcal{P}(\mathbf{R}^{d} \setminus \square_{R N^{-\frac{1}{d}}})}   \overline{\mathcal{I}}_{V}^{Y_{N}} (\rho) + \frac{1}{ \left( N - i_{N} \right) \beta} \text{ent}[\rho],
\end{equation}
where
\begin{equation}
\begin{split}
    \overline{\mathcal{I}}_{V}^{Y_{N}} (\rho)= \frac{1}{2}  \mathcal{E}_{\neq}({\rm emp}_{N}'(Y_{N}))+\frac{1}{2}  \left( 1 - \frac{i_{N}}{N} \right)^{2} \iint_{\mathbf{R}^{d} \times \mathbf{R}^{d}} g(x-y) d\rho_{x} d\rho_{y} +\\
    \left( 1 - \frac{i_{N}}{N} \right)\int \widetilde{V} d \rho.
\end{split}    
\end{equation}
The reason for introducing the extra factors of $1 - \frac{i_{N}}{N}$ is that we need to normalize the total charge outside the cube to be $1$. We then have that probability measures satisfy a splitting formula around $\mu_{\beta}^{Y_{N}}$, analogous to equation \eqref{eq:splittingform}.

   We also have that $ \mu_{\beta}^{Y_{N}} $ satisfies the EL equation
\begin{equation}\label{ELthermaleqmeasure}
     \left( 1 - \frac{i_{N}}{N} \right)^{2} h^{\mu_{\beta}^{Y_{N}}} +  \left( 1 - \frac{i_{N}}{N} \right)\widetilde{V} + \frac{1}{ \left( N - i_{N} \right) \beta} \log(\mu_{\beta}^{Y_{N}}) = k,
\end{equation}
for some constant $k.$ In order to find $k,$ we can multiply equation \eqref{ELthermaleqmeasure} by $\mu_{\beta}^{Y_{N}},$ integrate, and use that $\mu_{\beta}^{Y_{N}}$ has integral $1$. The result is 
\begin{equation}
    k = \overline{\mathcal{I}}_{V}^{Y_{N}} (\mu_{\beta}^{Y_{N}}) + \frac{1}{N \beta} \text{ent}[\mu_{\beta}^{Y_{N}}] + \frac{1}{2} \left( 1 - \frac{i_{N}}{N} \right) \mathcal{E}(\mu_{\beta}^{Y_{N}}). 
\end{equation} 

Plugging in $\mu =( \mu_{\beta}^{Y_{N}})^{\otimes N - i_{N}}$ as a test function in equation \eqref{eq:genchar} we get
\begin{equation} \label{variationalcharacter}
    \begin{split}
         -\frac{\log (Z_{N, \beta}^{Y_{N}})}{\beta } &\leq  \mathcal{F}(( \mu_{\beta}^{Y_{N}})^{\otimes N - i_{N}})\\
         & = N^{2} \left( \overline{\mathcal{I}}_{V}^{Y_{N}} (\mu_{\beta}^{Y_{N}}) + \frac{1}{ \left( N - i_{N} \right) \beta} \text{ent}[\mu_{\beta}^{Y_{N}}]  \right) - \left( \frac{N-i_{N}}{2} \right) \mathcal{E}(\mu_{\beta}^{Y_{N}})\\
         & \leq N^{2} \left( \min_{\rho} \overline{\mathcal{I}}_{V}^{Y_{N}} (\rho) + \frac{1}{ \left( N - i_{N} \right) \beta} \text{ent}[\rho]  \right) - C N .
    \end{split}
\end{equation}
The negative term of order $\left( N - i_{N} \right)$ is due to the fact that there are $\frac{(N-i_{N})(N-i_{N}-1)}{2}$ pair of particles, and not $\frac{(N-i_{N})^{2}}{2}$ pairs.

Finally, note that the last term in equation \eqref{variationalcharacter} is (up to a negligible error) equivalent to 
\begin{equation}
     \min_{\rho \in \mathcal{M}^{+}(\mathbf{R}^{d} \setminus \square_{R N^{-\frac{1}{d}}})}   \widetilde{\mathcal{I}}_{V}^{Y_{N}} (\rho) + \frac{1}{N\beta} \text{ent}[\rho],
\end{equation}
where
\begin{equation}
    \widetilde{\mathcal{I}}_{V}^{Y_{N}} (\rho)= \frac{1}{2} \mathcal{E}_{\neq}({\rm emp}_{N}'(Y_{N}))+ \frac{1}{2}  \iint_{\mathbf{R}^{d} \times \mathbf{R}^{d}} g(x-y) d\rho_{x} d\rho_{y} + \int \widetilde{V} d \rho,   
\end{equation}
and the minimum is taken over $\rho$ satisfying $|\rho| = 1 - \frac{i_{N}}{N}$.}

Using the splitting formula for the thermal equilibrium measure (Proposition \ref{thermalsplittingformula}) we have that 
\begin{equation}\label{formulaforpartitionfunc}
\begin{split}
    &-\log ( Z_{N, \beta}^{Y_{N}}) \leq N^{2} \beta \bigg( \mathcal{E}_{\beta}(\mu_{\beta})+\inf_{\rho} \frac{1}{2}\mathcal{E}^{\neq}_{\square_{R N^{-\lambda}}}({\rm emp}_{N}'(Y_{N})+\rho - \mu_{\beta}) - \\
    & \frac{1}{N \beta}\int_{\mathbf{R}^{d}} \log(\mu_{\beta}) d( \rho + {\rm emp}_{N}'(Y_{N}))+  \frac{1}{N \beta} {\rm ent}[\rho]  \bigg),
\end{split}
\end{equation}
where the infimum is taken over all measures $\rho$ on $\mathbf{R}^{d} \setminus \square_{RN^{-\lambda}}$ which are positive and such that 
\begin{equation}
    |\rho|=1-\frac{i_{N}}{N}.
\end{equation}

\textbf{\textit{Step 2}} 

This step is divided into two cases. The case $\gamma < \gamma^{*}$ and the case $\gamma > \gamma^{*}$. First we deal with the case $\gamma < \gamma^{*}$.

\textbf{\textit{Substep 2.1:}} Regime $\gamma < \gamma^{*}$.  

In this case, we claim that
\begin{equation}
    N^{\lambda(d+2)} \left( -\frac{1}{N^{2} \beta}\log (  Z_{N, \beta}^{Y_{N}}) -  \mathcal{E}_{\beta}(\mu_{\beta}) \right) \leq  \frac{1}{2}\mathbf{F}_{\square_{R}}^{\mu_{V}(0)}(\nu)+C \eta +o_{N}(1),
\end{equation}
where $ \mathbf{F}_{\square_{R}}^{\mu_{V}(0)}$ is given by \eqref{defofPhibeta}. The proof will be divided into $3$ further subsubsteps. 

\textbf{\textit{Subsubstep 2.1.1}}

We now begin the proof of the claim. Using the scaling relations
\begin{equation}
    {\rm ent}[\rho] = N^{-\lambda d }    {\rm ent}[\rho^{N^{\lambda}}] 
\end{equation}
and 
\begin{equation}
    \mathcal{E} (\rho) = N^{-\lambda(d+2)}  \mathcal{E} (\rho^{N^{\lambda}}),  
\end{equation}
we can rewrite equation \eqref{formulaforpartitionfunc} as 
\begin{equation} \label{formulaforpartitionfuncreescaled}
\begin{split}
    &-\frac{1}{N^{2} \beta}\log (  Z_{N, \beta}^{Y_{N}}) -  \mathcal{E}_{\beta}(\mu_{\beta}) \leq \\
    &\inf_{\rho} \Bigg( \frac{N^{-\lambda(d+2)}}{2} \mathcal{E}^{\neq}_{\square_{R}}({\rm lemp}_{N}(Y_{N})+\rho - \mu_{\beta}^{N^{\lambda}}) -\\
    &\frac{1}{N^{1+\lambda d} \beta}\int_{\mathbf{R}^{d}} \log(\mu_{\beta}^{N^{\lambda}}) d( \rho + {\rm lemp}_{N}(Y_{N}))+\frac{1}{\beta N^{1+\lambda d}} {\rm ent}[\rho] + C N^{-\frac{2}{d}} \Bigg),
\end{split}    
\end{equation}
where the infimum is taken over all $\rho \in \mathcal{M}^{+}(\mathbf{R}^{d} \setminus \square_{R})$ such that 
\begin{equation}
    |\rho| =  N^{\lambda d }\left( 1-\frac{i_{N}}{N} \right).
\end{equation}

Next we will argue that we can deal with $\nu$ instead of ${\rm lemp}_{N}(Y_{N})$ since we make a small error when approximating $\nu$ by ${\rm lemp}_{N}(Y_{N})$. This will be the subject of the next subsubstep.

\textbf{\textit{Substep 2.1.2}}

First of all, note that the constraint $ |\rho| =  N^{\lambda d }\left( 1-\frac{i_{N}}{N} \right)$ can be replaced by the constraint $ |\rho| =  N^{\lambda d } - |\nu|$ while making a negligible error because of equation \eqref{eq:massmatching}. We now introduce the functionals $\mathbf{L}^{*}$ and  $\mathbf{L}^{*}_{\neq}$, defined for measures $\nu$ and ${\rm lemp}_{N}$ on $\square_{R}$ as 
\begin{equation}\label{defofPhi*}
\begin{split}
    &\mathbf{L}^{*}(\nu)= \\
    &\inf_{\rho}  \frac{1}{2}\mathcal{E}(\nu+\rho - \mu_{\beta}^{N^{\lambda}}) - \frac{N^{\lambda(d+2)}}{N^{1+\lambda d} \beta}\int_{\mathbf{R}^{d}} \log(\mu_{\beta}^{N^{\lambda}}) d( \rho + \nu)+   \frac{N^{\lambda(d+2)}}{\beta N^{1+\lambda d}} {\rm ent}[\rho],
\end{split}    
\end{equation}
and
\begin{equation}\label{defofPhi*2}
\begin{split}
    &\mathbf{L}^{*}_{\neq}({\rm lemp}_{N})=\inf_{\rho}  \frac{1}{2}\mathcal{E}^{\neq}_{\square_{R}}( {\rm lemp}_{N}(Y_{N}))+\rho - \mu_{\beta}^{N^{\lambda}}) - \\
    & \frac{N^{\lambda(d+2)}}{N^{1+\lambda d} \beta}\int_{\mathbf{R}^{d}} \log(\mu_{\beta}^{N^{\lambda}}) d( \rho +  {\rm lemp}_{N}(Y_{N})))+   \frac{N^{\lambda(d+2)}}{\beta N^{1+\lambda d}} {\rm ent}[\rho],
\end{split}    
\end{equation}
where the infimum in equations \eqref{defofPhi*} and \eqref{defofPhi*2} is taken over all $\rho \in \mathcal{M}^{+}(\mathbf{R}^{d} \setminus \square_{R})$ such that 
\begin{equation}
    |\rho| = N^{\lambda d } - |\nu|.
\end{equation}

Given $\nu \in \mathcal{M}^{+}(\square_{R})$ and $Y_{N} \in \square_{R}^{i_{N}}$, let 
\begin{equation}
\begin{split}
    &\rho^{*}_{\nu}= \\
    &\operatorname{argmin}_{\rho}  \frac{1}{2}\mathcal{E}(\nu+\rho - \mu_{\beta}^{N^{\lambda}}) - \frac{N^{\lambda(d+2)}}{N^{1+\lambda d} \beta}\int_{\mathbf{R}^{d}} \log(\mu_{\beta}^{N^{\lambda}}) d( \rho + \nu)+   \frac{N^{\lambda(d+2)}}{\beta N^{1+\lambda d}} {\rm ent}[\rho],
\end{split}    
\end{equation}
and similarly, let $\rho^{*}_{Y_{N}}$ be defined as 
\begin{equation}
\begin{split}
    &\rho^{*}_{Y_{N}}= \inf_{\rho}  \frac{1}{2}\mathcal{E}^{\neq}_{\square_{R}}( {\rm lemp}_{N}(Y_{N}))+\rho - \mu_{\beta}^{N^{\lambda}}) -\\
    & \frac{N^{\lambda(d+2)}}{N^{1+\lambda d} \beta}\int_{\mathbf{R}^{d}} \log(\mu_{\beta}^{N^{\lambda}}) d( \rho +  {\rm lemp}_{N}(Y_{N})))+   \frac{N^{\lambda(d+2)}}{\beta N^{1+\lambda d}} {\rm ent}[\rho],
\end{split}    
\end{equation}
where the infimum is taken over all $\rho \in \mathcal{M}^{+}(\mathbf{R}^{d} \setminus \square_{R})$ such that 
\begin{equation}
    |\rho| = N^{\lambda d } - |\nu|.
\end{equation}
We assume that the infimum is achieved for clarity of exposition. Otherwise, we could repeat the argument up to an arbitrarily small error. 
Then we can use $\rho^{*}_{\nu}$ as a test function in equation \eqref{defofPhi*} and get
\begin{equation}
\begin{split}
    &\mathbf{L}^{*}(\nu) -  \mathbf{L}^{*}_{\neq}({\rm lemp}_{N}(Y_{N})) \leq\\
    &\frac{1}{2}\mathcal{E}(\nu) - \frac{1}{2} \mathcal{E}^{\neq}({\rm lemp}_{N}(Y_{N})) + G(\rho^{*}_{\nu} - \mu_{\beta}^{N^{\lambda}}, {\nu}- {{\rm lemp}_{N}(Y_{N})}) +\\
    & \frac{N^{\lambda(d+2)}}{N^{1+\lambda d} \beta}\int_{\mathbf{R}^{d}} \log(\mu_{\beta}^{N^{\lambda}}) d( {\rm lemp}_{N} - \nu).
\end{split}    
\end{equation}

Similarly, we can use $\rho^{*}_{Y_{N}}$ as test function in equation \eqref{defofPhi*} and get
\begin{equation}
\begin{split}
   &\mathbf{L}^{*}_{\neq}({\rm lemp}_{N}(Y_{N})) - \mathbf{L}^{*}(\nu)   \leq \\
   &\frac{1}{2}\mathcal{E}^{\neq}({\rm lemp}_{N}(Y_{N})) - \frac{1}{2}\mathcal{E}(\nu)  + G(\rho^{*}_{Y_{N}} - \mu_{\beta}^{N^{\lambda}} , {\nu}- {{\rm lemp}_{N}(Y_{N})}) -\\
    & \frac{N^{\lambda(d+2)}}{N^{1+\lambda d} \beta}\int_{\mathbf{R}^{d}} \log(\mu_{\beta}^{N^{\lambda}}) d( {\rm lemp}_{N} - \nu).
\end{split}   
\end{equation}

Using equation \eqref{muclosetoempN} we have that
\begin{equation}
    |\mathcal{E}^{\neq}({\rm lemp}_{N}(Y_{N})) - \mathcal{E}(\nu) | \leq C \eta,
\end{equation}
where $C$ depends on $\nu$. 

Then, since the points are at distance at least $r N^{-\frac{1}{d}}$ from $\partial \square_{R N^{-\lambda}}$, we have by Lemma \ref{smearinglemma1}, that, for $x \notin \square_{R}$
\begin{equation}
    g \ast {\rm lemp} (x) =  g \ast {\rm lemp} \ast \lambda_{r N^{\lambda - \frac{1}{d}}} (x),
\end{equation}
(see Remark \ref{def:lamb} for notation). 

Using Cauchy-Schwartz we get
\begin{equation}
    \begin{split}
       & \mathcal{G}(\rho^{*}_{\nu} - \mu_{\beta}^{N^{\lambda}}, {\nu}- {{\rm lemp}_{N}(Y_{N})}) =\\
        & \mathcal{G}(\rho^{*}_{\nu} - \mu_{\beta}^{N^{\lambda}}, {\nu}- {{\rm lemp}_{N}(Y_{N})}\ast \lambda_{r N^{\lambda-\frac{1}{d}}}) \leq \\
        &\sqrt{\mathcal{E}(\rho^{*}_{\nu} - \mu_{\beta}^{N^{\lambda}}) \mathcal{E}( {\nu}- {{\rm lemp}_{N}(Y_{N})}\ast \lambda_{r N^{\lambda-\frac{1}{d}}})}. 
    \end{split}
\end{equation}

Using now the hypothesis that $\nu$ has $L^{\infty}$ regularity, along with Lemmas \ref{smearinglemma1}, \ref{smearinglemma2},  \ref{smearinglemma3},  \ref{smearinglemma4} we have that
\begin{equation}
    \mathcal{E}( {\nu}- {{\rm lemp}_{N}(Y_{N})}\ast \lambda_{r N^{\lambda-\frac{1}{d}}}) \leq C \eta^{2},
\end{equation}
where $C$ depends on $\nu$ and $r$.

On the other hand, it is easy to see that 
\begin{equation}
    \mathcal{E}(\rho^{*}_{\nu} - \mu_{\beta}^{N^{\lambda}}) \leq C,
\end{equation}
where $C$ depends on $\nu$. Therefore 
\begin{equation}
        \mathcal{G}(\rho^{*}_{\nu} - \mu_{\beta}^{N^{\lambda}}, {\nu}- {{\rm lemp}_{N}(Y_{N})}) \leq C \eta,
\end{equation}
where $C$ depends on $\nu$ and $r$. Similarly,
\begin{equation}
        \mathcal{G}(\rho^{*}_{Y_{N}} - \mu_{\beta}^{N^{\lambda}}, {\nu}- {{\rm lemp}_{N}(Y_{N})}) \leq C \eta,
\end{equation}
where $C$ depends on $\nu$ and $r$.

This implies that
\begin{equation}
    \left| \mathbf{L}^{*}(\nu) - \mathbf{L}^{*}_{\neq}({\rm lemp}_{N}(Y_{N})) \right| \leq C \eta, 
\end{equation}
where $C$ depends on $\nu$.

We have proved that we can deal with $\nu$ instead of ${\rm lemp}_{N}(Y_{N})$ since we make a small error when approximating $\nu$ by ${\rm lemp}_{N}(Y_{N})$. The last subsubstep will consist in proving that
\begin{equation}
\label{eq:subsubstep}
        \limsup_{N \to \infty}\mathbf{L}^{*}(\nu)  \leq  \frac{1}{2} \mathbf{F}_{\square_{R}}^{\mu_{V}(0)}(\nu).
\end{equation}

\textbf{\textit{Substep 2.1.3}}

The proof of equation \eqref{eq:subsubstep} will consist in taking a minimizing sequence of the problem in the RHS, and modifying it so that it is a valid test function to the problem in the LHS. 

Let $\rho_{\epsilon} \geq 0$ be such that 
\begin{equation}
    \int_{\mathbf{R}^{d} } \rho_{\epsilon} + \nu - \mu_{V}(0) dx = 0
\end{equation}
and 
\begin{equation}
    \mathcal{E}(\nu+\rho_{\epsilon} - \mu_{V}(0)) \leq   \mathbf{F}_{\square_{R}}^{\mu_{V}(0)}(\nu) + \epsilon.
\end{equation}

Then for every $\delta > 0$ there exists $T>0$ such that
\begin{equation}\label{errorinmass}
   \left| \int_{B(0,T)} \rho_{\epsilon} + \nu - \mu_{V}(0) dx \right| \leq \delta
\end{equation}
and 
\begin{equation}
    \begin{split}\label{errorinenergy}
        &\left| \mathcal{E} \left(\nu+\rho_{\epsilon} - \mu_{V}(0) \right) - \mathcal{E}\left((\nu+\rho_{\epsilon} - \mu_{V}(0)\right)\mathbf{1}_{B(0,T)}) \right| \leq \delta.
    \end{split}
\end{equation}

Now take a truncated $\rho_{\epsilon}^{\overline{\eta}}$ such that \eqref{errorinenergy} and \eqref{errorinmass} hold with an error $\delta+\overline{\eta}$ in the right hand side, and in addition 
\begin{equation}
    \rho_{\epsilon}^{\overline{\eta}} \in L^{\infty}.
\end{equation}
Note that $\rho_{\epsilon}^{\overline{\eta}}$ exists because the sequence
\begin{equation}
     \rho_{\epsilon} \mathbf{1}_{|\rho_{\epsilon}|<M}
\end{equation}
is bounded, and by Dominated Convergence Theorem, its integral converges to the integral of $\rho_{\epsilon}$ as $M \to \infty$.

Now define $\rho_{\epsilon,T}^{\overline{\eta}}$ as 
\begin{equation}
    \rho_{\epsilon,T}^{\overline{\eta} } = \rho_{\epsilon}^{\overline{\eta}} \mathbf{1}_{B(0,T)} + \mu_{\beta}^{N^{\lambda}}.
\end{equation}

 Note that
\begin{equation}
      \left|   \int_{\mathbf{R}^{d}} -\log(\mu_{\beta}^{N^{\lambda}}) d( \rho_{\epsilon,T}^{\overline{\eta}} + \nu)+  {\rm ent}[\rho_{\epsilon,T}^{\overline{\eta}}] \right| \leq \\
      C ,
\end{equation}
where $C$ depends on $T$ and $\overline{\eta}$ but does not depend on $N.$ Since we are in the regime $\gamma < \gamma^{*},$ we have that 
\begin{equation}
    1+\lambda d - \gamma > \lambda(d+2),
\end{equation}
and therefore
\begin{equation}
    \lim_{N \to \infty} \frac{N^{\lambda(d+2)}}{N^{1+\lambda d}\beta}  \left|   \int_{\mathbf{R}^{d}} -\log(\mu_{\beta}^{N^{\lambda}}) d( \rho_{\epsilon,T}^{\overline{\eta}} + \nu)+  {\rm ent}[\rho_{\epsilon,T}^{\overline{\eta}}] \right|  =0.
\end{equation}

Using $\rho_{\epsilon,T}^{\overline{\eta}} $ as a test function in the definition of $\mathbf{L}^{*}(\nu)$, and appealing once again to Remark \ref{rem:unifconv} we have that
\begin{equation}
    \begin{split}
       & \limsup_{N \to \infty} \mathbf{L}^{*}(\nu)\\
        \leq & \limsup_{N \to \infty} \frac{1}{2}\mathcal{E} (\nu+ \rho_{\epsilon,T}^{\overline{\eta}} -\mu_{\beta}^{N^{\lambda}} ) +\\
        & \quad\frac{N^{\lambda(d+2)}}{N^{1+\lambda d}\beta}  \left|   \int_{\mathbf{R}^{d}} -\log(\mu_{\beta}^{N^{\lambda}}) d( \rho_{\epsilon,T}^{\overline{\eta}} + \nu)+  {\rm ent}[\rho_{\epsilon, T}^{\overline{\eta}}] \right|\\
        \leq   & \limsup_{N \to \infty} \frac{1}{2}\mathcal{E}((\nu+\rho_{\epsilon, R}^{\overline{\eta}} - \mu_{\beta}^{N^{\lambda}} ) \mathbf{1}_{B(0,T)}) \\
        = & \frac{1}{2}\mathcal{E}\left((\nu+\rho_{\epsilon, R}^{\overline{\eta}}- \mu_{V}(0))\right)\mathbf{1}_{B(0,T)})  \\
        \leq  &\frac{ \mathbf{F}_{\square_{R}}^{\mu_{V}(0)}(\nu)}{2}   + \epsilon+ \delta + \overline{\eta}.
    \end{split}
\end{equation}

Since $\epsilon, \delta, \overline{\eta}$ are arbitrary, we conclude
\begin{equation}
        \limsup_{N \to \infty}\mathbf{L}^{*}(\nu)  \leq \frac{1}{2}   \mathbf{F}_{\square_{R}}^{\mu_{V}(0)}(\nu).
\end{equation}

The proof of substep 2.1 is now complete. 

\textbf{\textit{Substep 2.2}}  Now we deal with the case $\gamma > \gamma^{*}.$ In this case we go back to working in unreescaled coordinates.

We start with formula \eqref{formulaforpartitionfunc}. Since in the regime $\gamma > \gamma^{*},$ we expect the term
\begin{equation}
 \mathcal{E}^{\neq}_{\square_{R}}({\rm emp}_{N}'(Y_{N})+\rho - \mu_{\beta}) 
\end{equation}
to be negligible, we focus on the remaining part of the functional, i.e.
\begin{equation}\label{problementropypotential}
    \inf_{\rho} - \int_{\mathbf{R}^{d}} \log(\mu_{\beta}) d( \rho + {\rm emp}_{N}'(Y_{N}))+  {\rm ent}[\rho],
\end{equation}
where the infimum is taken over all measures $\rho$ on $\mathbf{R}^{d} \setminus \square_{RN^{-\lambda}}$ which are positive and such that 
\begin{equation}
    |\rho|=1-\frac{i_{N}}{N}.
\end{equation}
The minimizer in equation \eqref{problementropypotential} can be easily found by adding a Lagrange multiplier for the mass constraint.
It can be easily checked that the unique minimizer of \eqref{problementropypotential} in the corresponding space is given by $\rho^{*}$, where
\begin{equation}
    \rho^{*} = \alpha \mu_{\beta} \mathbf{1}_{\mathbf{R}^{d} \setminus \square_{R N^{-\lambda}}},
\end{equation}
where
\begin{equation}
    \alpha= \frac{ 1-\frac{i_{N}}{N}}{\int_{\mathbf{R}^{d} \setminus \square_{R N^{-\lambda}}} \mu_{\beta} dx}.
\end{equation}

Using the identity
\begin{equation}
        {\rm ent}[A \mu] = A {\rm ent}[ \mu]+ A \log(A) \int  \mu dx,
\end{equation}
valid for any $A \in \mathbf{R}^{+}$, we have that
\begin{equation}
\begin{split}
    &  \int_{\mathbf{R}^{d}} \log(\mu_{\beta}) d( \rho^{*} + {\rm emp}_{N}'(Y_{N}))+ {\rm ent}[\rho^{*}] =  \\
   &   \int_{\mathbf{R}^{d}} \log(\mu_{\beta}) d( {\rm emp}_{N}'(Y_{N})) + \left( 1-\frac{i_{N}}{N} \right) \log(\alpha) .
\end{split}    
\end{equation}

It can be checked that, { as a consequence of equation \eqref{eq:massmatching}}, $\lim_{N \to \infty} \alpha =1$, and therefore we may use the approximation $\log \alpha \simeq \alpha-1$. Proceeding as in the proof of Lemma \ref{entropyterm} and using equation \eqref{eq:massmatching}, we have that 
\begin{equation}
  \lim_{N \to \infty} N^{\lambda d} \left( 1-\frac{i_{N}}{N} \right) \log(\alpha) = R^{d} \mu_{V}(0) - |\nu|.
\end{equation}
Note that 
\begin{equation}
 \lim_{N \to 0} \frac{N^{1 + \lambda d } \beta}{N^{\lambda(d+2)}}  \mathcal{E}(\rho^{*} - \mu_{\beta}) =0
\end{equation}
since we are in the regime $\gamma > \gamma^{*}$. Using again formula \eqref{formulaforpartitionfunc} and switching to rescaled coordinates, we have
\begin{equation}
\begin{split}
& N^{1+\lambda d} \beta \left(  -\frac{1}{ N^{2} \beta}\log ( Z_{N, \beta}^{Y_{N}}) - \mathcal{E}_{\beta}(\mu_{\beta}) \right) \leq  \\
& \quad  \int_{\mathbf{R}^{d}} \log(\mu_{\beta}^{N^{\lambda}}) d( {\rm lemp}_{N}(Y_{N})) -|\nu|+R^{d} \mu_{V}(0) + o_{N}(1),
\end{split} 
\end{equation}
where $o(1)$ is independent of $Y_{N}.$

Lemma \ref{boundonZYN} is proved for $\gamma > \gamma^{*}$.

\textbf{\textit{Step 3}}  

This step only deals with the case $\gamma < \gamma^{*}.$ Once again we work with rescaled coordinates. 

We now claim that for any measure $\nu$ on $\square_{R}$ such that $\mathcal{E}(\nu) < \infty$ we have
\begin{equation}
      \mathbf{F}_{\square_{R}}^{\mu_{V}(0)}(\nu) =  \Phi_{\square_{R}}^{\mu_{V}(0)}(\nu) .
\end{equation}
In other words, we claim that we can drop the mass constraint. We now prove the claim. Since clearly
\begin{equation}
      \mathbf{F}_{\square_{R}}^{\mu_{V}(0)}(\nu) \geq  \Phi_{\square_{R}}^{\mu_{V}(0)}(\nu) , 
\end{equation}
we will prove that
\begin{equation}
      \mathbf{F}_{\square_{R}}^{\mu_{V}(0)}(\nu) \leq  \Phi_{\square_{R}}^{\mu_{V}(0)}(\nu). 
\end{equation}
In order to prove this claim, we reformulate the definition of $\Phi_{\square_{R}}^{\mu_{V}(0)}(\nu)$ as 
as 
\begin{equation}
   \Phi_{\square_{R}}^{\mu_{V}(0)}( \nu) = \inf_{\rho} \mathcal{E} (\nu-\mu_{V}(0) \mathbf{1}_{\square_{R}} + \rho)
\end{equation}
where the infimum is taken over all $\rho$ such that $\rho$ is supported in $\mathbf{R}^{d} \setminus \square_{R}$ and $\rho \geq -\mu_{V}(0).$

Let $\epsilon > 0$ and let $\rho_{\epsilon} \in C^{\infty}_{0}$ be such that $\rho_{\epsilon}$ is supported in $\mathbf{R}^{d} \setminus \square_{R}$, $\rho_{\epsilon} \geq -\mu_{V}(0)$ and 
\begin{equation}
     \mathcal{E} (\nu-\mu_{V}(0) \mathbf{1}_{\square_{R}} + \rho_{\epsilon}) \leq \Phi_{\square_{R}}^{\mu_{V}(0)}( \nu) + \epsilon. 
\end{equation}
Let 
\begin{equation}
    K = \text{supp} (\rho_{\epsilon}),
\end{equation}
and let 
\begin{equation}
    E =  |\rho_{\epsilon}| - (|\nu| - R^{d} \mu_{V}(0)).
\end{equation}
Let $R_{n}$ be a sequence such that $R_{n}$ tends to $\infty$ monotonically, and \begin{equation}
    K \subset B(0,R_{1}).
\end{equation}
Define 
\begin{equation}
    \rho_{\epsilon}^{n} = \rho_{\epsilon} + \frac{E}{\mathcal{L} (B(0,2 R_{n}) \setminus B(0,R_{n}))} \mathbf{1}_{B(0,2 R_{n}) \setminus B(0,R_{n})},
\end{equation}
where $\mathcal{L}$ denotes the Lebesgue measure. Then it's easy to see that
\begin{equation}
     | \rho_{\epsilon}^{n} | = - (|\nu| - R^{d} \mu_{V}(0)),
\end{equation}
and 
\begin{equation}
    \lim_{n \to \infty}  \mathcal{E} (\nu-\mu_{V}(0) \mathbf{1}_{\square_{R}} + \rho_{\epsilon}^{n}) =  \mathcal{E} (\nu-\mu_{V}(0) \mathbf{1}_{\square_{R}} + \rho_{\epsilon}).
\end{equation}
Therefore
\begin{equation}
      \mathbf{F}_{\square_{R}}^{\mu_{V}(0)}(\nu) \leq  \Phi_{\square_{R}}^{\mu_{V}(0)}(\nu) + \epsilon. 
\end{equation}
Since $\epsilon$ is arbitrary, we conclude that
\begin{equation}
      \mathbf{F}_{\square_{R}}^{\mu_{V}(0)}(\nu) =  \Phi_{\square_{R}}^{\mu_{V}(0)}(\nu). 
\end{equation}

\end{proof}

We can now complete the proof of Theorem \ref{maintheorem} by giving the lower bound.

\begin{proof}[Proof of Theorem \ref{maintheorem}, lower bound]
We start with the case $\gamma < \gamma^{*}.$ Let $\nu$ be a measure on $\square_{R}.$ Using a density argument, we may assume that $\nu \in L^{\infty}$ and ${\rm ent}[\nu] < \infty.$ Let $\epsilon, \eta, \delta
>0$ and let $\Lambda_{\delta}^{\eta, \epsilon}$ be as in Lemma \ref{goodenergygoodvolume} with $\Omega = \square_{R}$ and $\mu_{N} = \frac{\mu_{\beta}^{N^{\lambda}}|_{\square_{R}}}{|\mu_{\beta}^{N^{\lambda}}|_{\square_{R}}|}$, $n= |\nu| N^{1 - \lambda d}$ (rounded to an integer) and $\overline{\nu}=\frac{\nu}{|\nu|}$. Note that equation \eqref{muclosetoempN} is satisfied with this choice of $n$. We claim that equation \eqref{goodenergy} implies that there exists $C$ which depends on $\nu$ such that for any $X_{N} \in \Lambda_{\delta}^{\eta, \epsilon}$, 
\begin{equation}
    \Phi^{\mu_{V}(0)}_{\square_{R}, \neq} ({\rm emp}_{N}(X_{N})) \leq  \Phi^{\mu_{V}(0)}_{\square_{R}}(\nu) + C \eta.
\end{equation}
This is because for 
\begin{equation}
    \rho^{*} = \operatorname{argmin}_{\rho \in \mathcal{M}^{+}(\mathbf{R}^{d} \setminus \square_{R})} \mathcal{E}(\nu+\rho -  \mu_{V}(0))
\end{equation}
we have 
\begin{equation}
    \begin{split}
        \Phi^{\mu_{V}(0)}_{\square_{R}, \neq} ({\rm emp}_{N}(X_{N})) &\leq \mathcal{E}^{\neq}({\rm emp}_{N}(X_{N})+\rho^{*} -  \mu_{V}(0))\\
        &= \mathcal{E}^{\neq}({\rm emp}_{N}(X_{N})+ \nu - \nu +\rho^{*} -  \mu_{V}(0))\\
        &=  \mathcal{E}(\nu+\rho^{*} -  \mu_{V}(0))+\\
        &\quad 2\mathcal{G}( \nu+\rho^{*} -  \mu_{V}(0), {\rm emp}_{N}(X_{N})- \nu  ) + \\
        &\quad \mathcal{E}^{\neq}({\rm emp}_{N}(X_{N})- \nu)\\
        &\leq \Phi_{\square_{R}}^{\mu_{V}(0)}(\nu) + C \eta + \eta^{2},
    \end{split}
\end{equation}

where $C$ depends on $\nu$. Using Lemma \ref{boundonZYN} and equation \eqref{nextorderpartitionfunctionissmall} we then have that

\begin{equation}
    \begin{split}
        &\mathbf{P}_{N, \beta}({\rm lemp}_{N}(Y_{N}) \in B(\nu, \epsilon)) \geq \\
        &\mathbf{P}_{N, \beta}(Y_{N}^{\lambda} \in  \Lambda_{\delta}^{\eta, \epsilon})=\\
        & \frac{1}{Z_{N,\beta}} \int_{Y_{N} \in  \Lambda_{\delta}^{\eta, \epsilon}} Z_{N, \beta}^{Y_{N}} d Y_{N} \geq\\
        & \int_{Y_{N} \in  \Lambda_{\delta}^{\eta, \epsilon}} \exp\left( - \beta N^{2 - \lambda(d+2)} \left[ \frac{1}{2}\Phi_{\square_{R}}^{\mu_{V}(0)} (\nu) + C\eta + o_{N}(1) \right]  \right) d Y_{N} =\\
        &  \exp\left( - \beta N^{2 - \lambda(d+2)} \left[\frac{1}{2}\Phi_{\square_{R}}^{\mu_{V}(0)} (\nu) + C\eta + o_{N}(1)\right]  \right) \int_{Y_{N} \in  \Lambda_{\delta}^{\eta, \epsilon}}  d Y_{N}.
    \end{split}
\end{equation}

Since we are in the regime $\gamma < \gamma^{*}$ we have
\begin{equation}
    \begin{split}
        \left| \log \left( \int_{Y_{N} \in  \Lambda_{\delta}^{\eta, \epsilon}}  d Y_{N} \right) \right| &\leq C N^{1 - \lambda d} {\rm ent}[\nu]\\
        &= o(\beta N^{2 - \lambda(d+2)}).
    \end{split}
\end{equation}

Therefore
\begin{equation}
   \liminf_{N \to \infty} \frac{1}{\beta N^{2 - \lambda(d+2)}} \log \left( \mathbf{P}_{N, \beta}({\rm lemp}_{N}(Y_{N}) \in B(\nu, \epsilon)) \right) \geq  - \frac{1}{2} \Phi_{\square_{R}}^{\mu_{V}(0)} (\nu)-C \eta.
\end{equation}
Since $\eta$ is arbitrary, we can conclude that 
\begin{equation}
  \liminf_{N \to \infty} \frac{1}{\beta N^{2 - \lambda(d+2)}} \log \left( \mathbf{P}_{N, \beta}({\rm lemp}_{N}(Y_{N}) \in B(\nu, \epsilon)) \right) \geq  -  \frac{1}{2}\Phi_{\square_{R}}^{\mu_{V}(0)} (\nu).
\end{equation}

Now we proceed with the case $\gamma > \gamma^{*}.$ Let $\nu$ be a positive measure in $\square_{R}$, let $\epsilon, \eta, \delta
>0$ and let $\Lambda_{\delta}^{\eta, \epsilon}$ be as in Lemma \ref{goodenergygoodvolume} with $\Omega = \square_{R}$ and $\mu_{N} = \frac{\mu_{\beta}^{N^{\lambda}}|_{\square_{R}}}{|\mu_{\beta}^{N^{\lambda}}|_{\square_{R}}|}$, $n= |\nu| N^{1 - \lambda d}$ (rounded to an integer) and $\overline{\nu}=\frac{\nu}{|\nu|}$. Then, starting as in the previous case, we have
\begin{equation}
    \begin{split}
        \mathbf{P}_{N, \beta}({\rm lemp}_{N}(Y_{N}) \in B(\nu, \epsilon)) &\geq  \mathbf{P}_{N, \beta}(Y_{N}^{\lambda} \in  \Lambda_{\delta}^{\eta, \epsilon})\\
        &= \frac{1}{Z_{N, \beta}} \int_{Y_{N} \in  \Lambda_{\delta}^{\eta, \epsilon}} Z_{N, \beta}^{Y_{N}} d Y_{N}.
    \end{split}
\end{equation}    

We then have that
\begin{equation}
    \begin{split}
        &\liminf_{N \to \infty} \frac{1}{N^{1-\lambda d}} \log \left( \mathbf{P}_{N, \beta}({\rm lemp}_{N}(Y_{N}) \in B(\nu, \epsilon)) \right) \geq \\
        &\liminf_{N \to \infty} \frac{1}{Z_{N,\beta}} \frac{1}{N^{1-\lambda d}} \log \left(   \int_{Y_{N} \in  \Lambda_{\delta}^{\eta, \epsilon}} Z_{N, \beta}^{Y_{N}} d Y_{N} \right) =\\
        & \liminf_{N \to \infty} \frac{1}{N^{1-\lambda d}} \log\\
        &\quad  \left(  \int_{Y_{N} \in  \Lambda_{\delta}^{\eta, \epsilon}} \exp \left( - N^{1 - \lambda d} \left[\int_{\square_{R}} \log(\mu_{\beta}^{N^{\lambda}}) d \, {\rm lemp}_{N} \right) d Y_{N} - |\nu|+R^{d} \mu_{V}(0) \right]\right).
    \end{split}
\end{equation}

Recalling that
\begin{equation}
     {\rm lemp}_{N}(Y_{N}) = \frac{1}{N^{1 - \lambda d}} \sum_{i=1}^{i_{N}} \delta_{y_{i}^{\lambda}},
\end{equation}
and $i_{N} = N^{1 - \lambda d}|\nu|$ (rounded to an integer), we have that 
\begin{equation}
    \begin{split}
        &\liminf_{N \to \infty} \frac{1}{N^{1-\lambda d}} \log \left( \mathbf{P}_{N, \beta}({\rm lemp}_{N}(Y_{N}) \in B(\nu, \epsilon)) \right) \geq \\
        & \liminf_{N \to \infty}  \frac{1}{N^{1-\lambda d}} \log\\
        &\quad \left(  \int_{Y_{N} \in  \Lambda_{\delta}^{\eta, \epsilon}} \exp \left( - N^{1 - \lambda d} \left[\int_{\square_{R}} \log(\mu_{\beta}^{N^{\lambda}}) d \, {\rm lemp}_{N} \right) d Y_{N} -|\nu|+R^{d} \mu_{V}(0) \right]\right)=\\
        &  |\nu|-R^{d} \mu_{V}(0) + \liminf_{N \to \infty} \frac{1}{N^{1-\lambda d}} \log \left( \int_{Y_{N} \in  \Lambda_{\delta}^{\eta, \epsilon}}  \Pi_{i=1}^{i_{N}} \mu_{\beta}^{N^{\lambda}}(y_{i}) dY_{N}  \right).
    \end{split}
\end{equation}

By construction we have that 
 \begin{equation}
    \liminf_{N \to \infty} \frac{1}{N^{1-\lambda d}} \log \left( \int_{Y_{N} \in  \Lambda_{\delta}^{\eta, \epsilon}} \Pi_{i=1}^{i_{N}} \mu_{\beta}^{N^{\lambda}}(y_{i}) dy_{{i}} \right) \geq -  \liminf_{N \to \infty}{\rm ent}[\nu|\mu_{\beta}^{N^{\lambda}}] - \delta. 
 \end{equation}
Combining the last equation with Remark \ref{rem:unifconv} we have
 \begin{equation}
     \liminf_{N \to \infty} \frac{1}{N^{1-\lambda d}} \log \left( \int_{Y_{N} \in  \Lambda_{\delta}^{\eta, \epsilon}} \Pi_{i=1}^{i_{N}} \mu_{\beta}^{N^{\lambda}}(y_{i}) dy_{{i}} \right) \geq -{\rm ent}[\nu|\mu_{V}(0) \mathbf{1}_{\square_{R}}] - \delta, 
 \end{equation}
and therefore
\begin{equation}
\begin{split}
        &\liminf_{N \to \infty} \frac{1}{N^{1-\lambda d}} \log \left(  \mathbf{P}_{N, \beta} \left({\rm lemp}_{N}(Y_{N}) \in B(\nu, \epsilon) \right) \right) \geq \\
        &-{\rm ent}[\nu|\mu_{V}(0)\mathbf{1}_{\square_{R}}]  +|\nu|-R^{d} \mu_{V}(0)  - \delta. 
\end{split}        
\end{equation}

Since $\delta$ is arbitrary, we can conclude.  

\end{proof}

\section{Proof of statement about regime $\gamma = \gamma^{*}$}

In this section, we prove the third part of Theorem \ref{maintheorem}, which we repeat here for convenience: 
If $\gamma=\gamma^{*}$ (critical regime) and $\nu \in L^{\infty}$ then
     \begin{equation}
         \lim_{\epsilon \to 0} \limsup_{N \to \infty} \left( \frac{1}{\beta N^{2 - \lambda(d+2)} } \log \left( \mathbf{P}_{N,\beta} ({\rm lemp}_{N} \in {B}(\nu, \epsilon)) \right) +\mathcal{T}^{N}_{\lambda} (\nu) \right) = 0.
     \end{equation}
     Similarly,
     \begin{equation}
         \lim_{\epsilon \to 0} \liminf_{N \to \infty} \left( \frac{1}{\beta N^{2 - \lambda(d+2)} } \log \left( \mathbf{P}_{N,\beta} ({\rm lemp}_{N} \in  {B}(\nu, \epsilon)) \right) +\mathcal{T}^{N}_{\lambda} (\nu) \right) = 0.
     \end{equation}
     
Before starting the proof, we note that since we are in the critical regime, we have $\beta N^{2 - \lambda(d+2)} = N^{1 + \lambda d}$.      
     
\begin{proof}[Proof of $\liminf$ inequality]

 Let $\nu$ be a positive measure in $\square_{R}$, let $\epsilon, \eta, \delta
>0$ and let $\Lambda_{\delta}^{\eta, \epsilon}$ be as in Lemma \ref{goodenergygoodvolume} with $\Omega = \square_{R}$ and $\mu_{N} = \frac{\mu_{\beta}^{N^{\lambda}}|_{\square_{R}}}{|\mu_{\beta}^{N^{\lambda}}|_{\square_{R}}|}$, $n= |\nu| N^{1 - \lambda d}$ (rounded to an integer) and $\overline{\nu}=\frac{\nu}{|\nu|}$. Then we have that
\begin{equation}\label{firstestimategamma=gamma*}
    \frac{1}{\beta N^{2 - \lambda(d+2)}}\log\left(\mathbf{P}_{N, \beta} \left( {\rm lemp}_{N} \in  {B}(\nu, \epsilon) \right)\right) \geq  \frac{1}{\beta N^{2 - \lambda(d+2)}} \log \left( \int_{\Lambda_{\delta}^{\eta, \epsilon}} Z_{N,\beta}^{Y_{N}} dY_{N} \right).
\end{equation}

Using equation \eqref{formulaforpartitionfuncreescaled}, and the hypothesis that $\gamma = \gamma^{*},$ we have
\begin{equation}
\label{eq:critregime}
\begin{split}
    &-\frac{1}{N^{2} \beta}\log (  Z_{N, \beta}^{Y_{N}}) -  \mathcal{E}_{\beta}(\mu_{\beta}) \leq \\
    &\inf_{\rho} N^{-\lambda(d+2)} \Bigg( \frac{1}{2}\mathcal{E}^{\neq}({\rm lemp}_{N}(Y_{N})+\rho - \mu_{\beta}^{N^{\lambda}}) -\\
    &  \int_{\mathbf{R}^{d}} \log(\mu_{\beta}^{N^{\lambda}}) d( \rho + {\rm lemp}_{N}(Y_{N}))+{\rm ent}[\rho] \Bigg) \leq \\
    &  N^{-\lambda(d+2)} \Bigg( \mathbf{T}^{N}_{\lambda} (\nu) + \int_{\mathbf{R}^{d}} \log(\mu_{\beta}^{N^{\lambda}}) d({\rm lemp}_{N}(Y_{N})) -C \eta \Bigg),
\end{split}    
\end{equation}
where the infimum is taken over $\rho$ such that
\begin{equation}
    \int_{\mathbf{R}^{d}} {\rm lemp}_{N} + \rho - \mu_{\beta}^{N^{\lambda}}=0,
\end{equation}
and $\mathbf{T}^{N}_{\lambda}, \mathbf{T}^{N, \neq}_{\lambda}$ are given by equations \eqref{defofTN} and \eqref{defofT} respectively. We have used that, if $\nu \in L^{\infty}$ then 
\begin{equation}
   \left| \mathbf{T}^{N}_{\lambda} (\nu) -\mathbf{T}^{N, \neq}_{\lambda} ({\rm lemp}_{N}(Y_{N})) \right| \leq C \eta,
\end{equation}
where $C$ depends on $\nu.$ The proof of this statement is the same as the proof that
\begin{equation}
    \left| \mathbf{L}^{*}(\nu) - \mathbf{L}^{*}_{\neq}({\rm lemp}_{N}(Y_{N})) \right| \leq C \eta, 
\end{equation}
where $C$ depends on $\nu$, see the proof of Lemma \ref{boundonZYN}, step 2.1 (in fact, in the critical regime, we have that $\mathbf{T}^{N}_{\lambda} = \mathbf{L}^{*}$ and $\mathbf{T}^{N, \neq}_{\lambda} = \mathbf{L}^{*}_{\neq}$).

Therefore we can rewrite equation \eqref{eq:critregime} as 
\begin{equation}
\begin{split}
    &\liminf_{N \to \infty} \frac{1}{\beta N^{2 - \lambda(d+2)}}\log\left(\mathbf{P}_{N, \beta} \left( {\rm lemp}_{N} \in  {B}(\nu, \epsilon) \right)\right)  \\
    &\geq \liminf_{N \to \infty} \frac{1}{\beta N^{2 - \lambda(d+2)}} \log \left( \int_{\Lambda_{\delta}^{\eta, \epsilon}} Z_{N,\beta}^{Y_{N}} dY_{N} \right)\\
    &  \geq \liminf_{N \to \infty}  \frac{1}{\beta N^{2 - \lambda(d+2)}} \log \Bigg( \int_{\Lambda_{\delta}^{\eta, \epsilon}} \exp \Bigg( -\beta N^{ 2-\lambda(d+2)} \Bigg( \mathbf{T}^{N}_{\lambda} (\nu) +C\eta+\\
    & \quad \int_{\mathbf{R}^{d}} \log(\mu_{\beta}^{N^{\lambda}}) d({\rm lemp}_{N}(Y_{N})) \Bigg) \Bigg) dY_{N} \Bigg) \\
     &=\liminf_{N \to \infty}   \frac{1}{\beta N^{2 - \lambda(d+2)}} \log \\
     &\quad \left( \int_{\Lambda_{\delta}^{\eta, \epsilon}} \exp \left( -\beta N^{2-\lambda(d+2)} \Bigg( \mathbf{T}^{N}_{\lambda} (\nu) +C\eta  \Bigg) \right) \Pi_{i=1}^{i_{N}} \mu_{\beta}^{N^{\lambda}}(y_{i}) d y_{i} \right) \\
     & =\liminf_{N \to \infty}  \frac{1}{\beta N^{2 - \lambda(d+2)}} \log \left( \exp \left( -\beta N^{2-\lambda(d+2)}  \mathbf{T}^{N}_{\lambda} (\nu)+C\eta  \right) \right)-\\
     &\quad \mbox{ent}[\nu| \mu_{V}(0) \mathbf{1}_{\square_{R}} ] - \delta \\
     & = \liminf_{N \to \infty} - \mathcal{T}^{N}_{\lambda} (\nu) - C(\delta + \eta ), 
\end{split}    
\end{equation}
where $C$ depends on $\nu$.

Since $\eta$ and $\delta$ are arbitrary, we have 
\begin{equation}
    \liminf_{N \to \infty} \left( \frac{1}{\beta N^{2 - \lambda(d+2)}}\log \left( \mathbf{P}_{N,\beta} ({\rm lemp}_{N} \in  {B}(\nu, \epsilon)) \right) + \mathcal{T}^{N}_{\lambda} (\nu) \right) \geq 0.
\end{equation}

In particular, this implies
\begin{equation}
 \lim_{\epsilon \to 0} \liminf_{N \to \infty} \Bigg(\frac{1}{\beta N^{2 - \lambda(d+2)}} \log \left(\mathbf{P}_{N, \beta} \left( {\rm lemp}_{N} \in  {B}(\nu, \epsilon) \right) \right) + \\
     \left(  \mathcal{T}^{N}_{\lambda} (\nu) \right) \Bigg) \geq 0. 
\end{equation}

\end{proof}     
We now turn to the proof of the $\limsup$ inequality:

\begin{proof}[Proof of $\limsup$ inequality]

We start with equation \eqref{formulaforpartitionfuncreescaled}, which in the critical regime $\gamma = \gamma^{*}$ reads
\begin{equation}
\begin{split}
    &-\frac{1}{N^{2} \beta}\log (  Z_{N, \beta}^{Y_{N}}) -  \mathcal{E}_{\beta}(\mu_{\beta})  \\
    =&N^{-\lambda(d+2)} \inf_{\rho} \Bigg(  \frac{1}{2}\mathcal{E}^{\neq}_{\square_{R}}({\rm lemp}_{N}(Y_{N})+\rho - \mu_{\beta}^{N^{\lambda}}) -\\
    &\quad \int_{\mathbf{R}^{d}} \log(\mu_{\beta}^{N^{\lambda}}) d( \rho + {\rm lemp}_{N}(Y_{N}))+{\rm ent}[\rho] +o_{N}(1) \Bigg)=\\
    &N^{-\lambda(d+2)} \left( \mathbf{T}^{N \neq }_{\lambda}({\rm lemp}_{N}(Y_{N})) + \int_{\mathbf{R}^{d}} \log(\mu_{\beta}^{N^{\lambda}}) d( {\rm lemp}_{N}(Y_{N})) +o_{N}(1) \right),
\end{split}    
\end{equation}
where the infimum in line $2$ of the last equation is taken over all $\rho \in \mathcal{M}^{+}(\mathbf{R}^{d} \setminus \square_{R})$ such that 
\begin{equation}
    |\rho| =  N^{\lambda d }\left( 1-\frac{i_{N}}{N} \right).
\end{equation}

We proceed by writing 
\begin{equation}
    \begin{split}
        &\mathbf{P}_{N,\beta} ({\rm lemp}_{N} \in  {B}(\nu, \epsilon)) = \\
        &\int_{{\rm lemp}_{N} \in  {B}(\nu, \epsilon)} Z_{N,\beta}^{Y_{N}} dY_{N} =\\
        &\int_{{\rm lemp}_{N} \in  {B}(\nu, \epsilon)} \exp \Bigg( -\beta N^{2-\lambda(d+2)}[ \mathbf{T}^{N \neq }_{\lambda}({\rm lemp}_{N}(Y_{N})) + \\
        &\quad \int \log(\mu_{\beta}^{N^{\lambda}}) d({\rm lemp}_{N}(Y_{N})) ] \Bigg) dY_{N} =\\
        &\int_{{\rm lemp}_{N} \in  {B}(\nu, \epsilon)} \exp \left( -\beta N^{2-\lambda(d+2)} \mathbf{T}^{N \neq }_{\lambda}({\rm lemp}_{N}(Y_{N})) \right) \Pi_{i=1}^{i_{N}} \mu_{\beta}^{N^{\lambda}}(y_{i}) dy_{i} \leq \\
        & \exp \left( -\beta N^{2-\lambda(d+2)} \inf_{\mu \in  {B}(\nu, \epsilon) \cap \mathcal{A}_{i_{N}}^{N^{\lambda d -1}}(\square_{R})} \mathbf{T}^{N,  \neq  }_{\lambda}(\mu) \right)\\
        &\quad \int_{{\rm lemp}_{N} \in  {B}(\nu, \epsilon)} \Pi_{i=1}^{i_{N}} \mu_{\beta}^{N^{\lambda}}(y_{i}) dy_{i} \leq \\
        & \exp \left( -\beta N^{2-\lambda(d+2)} \left[ \inf_{\mu \in  {B}(\nu, \epsilon)  \cap \mathcal{A}_{i_{N}}^{N^{\lambda d -1}}(\square_{R})} \mathbf{T}^{N  \neq }_{\lambda}(\mu) + \inf_{\mu \in  {B}(\nu, \epsilon) } \mbox{ent}[\mu| \mu_{\beta}^{N^{\lambda}}\mathbf{1}_{\square_{R}}] \right] \right).
    \end{split}
\end{equation}

Letting $N$ tend to $\infty$ and using Remark \ref{rem:unifconv} we have that 
\begin{equation}
\begin{split}
    \limsup_{N \to \infty} \Bigg(\frac{1}{\beta N^{2 - \lambda(d+2)}}\log \left(\mathbf{P}_{N, \beta} \left( {\rm lemp}_{N} \in  {B}(\nu, \epsilon) \right) \right) + \\
    \left(  \inf_{\mu \in  {B}(\nu, \epsilon)  \cap \mathcal{A}_{i_{N}}^{N^{\lambda d -1}}(\square_{R})} \mathbf{T}^{N, \neq}_{\lambda}(\mu) + \inf_{\mu \in  {B}(\nu, \epsilon)} \mbox{ent}[\mu| \mu_{V}(0)\mathbf{1}_{\square_{R}}]  \right) \Bigg) \leq 0.
\end{split}    
\end{equation}

It's well known that ent$[\nu|\mu]$ is l.s.c. in $\nu$ for fixed $\mu.$ Therefore
\begin{equation}
    \lim_{\epsilon \to 0}  \inf_{\mu \in  {B}(\nu, \epsilon)} \mbox{ent}[\mu| \mu_{V}(0)\mathbf{1}_{\square_{R}}] =  \mbox{ent}[\nu| \mu_{V}(0)\mathbf{1}_{\square_{R}}]. 
\end{equation}

We will also use a property of $\mathbf{T}^{N, \neq}_{\lambda},$ which we prove at the end of this section: we will show that
\begin{equation}\label{tehnicalproperty}
    \lim_{\epsilon \to 0} \lim_{N \to \infty} \left( \mathbf{T}^{N}_{\lambda}(\nu) - \inf_{\mu \in  {B}(\nu, \epsilon)  \cap \mathcal{A}_{i_{N}}^{N^{\lambda d -1}}(\square_{R})}  \mathbf{T}^{N, \neq }_{\lambda}(\mu) \right) =0.
\end{equation}

Using equation \eqref{tehnicalproperty}, we have
\begin{equation}
         \lim_{\epsilon \to 0} \limsup_{N \to \infty} \Bigg(\frac{1}{ \beta N^{2 - \lambda(d+2)}}\log \left(\mathbf{P}_{N, \beta} \left( {\rm lemp}_{N} \in  {B}(\nu, \epsilon) \right) \right) + \mathcal{T}^{N}_{\lambda} (\nu) \Bigg) \leq 0. 
\end{equation}
This concludes the proof.
\end{proof}

We now prove equation \eqref{tehnicalproperty}, used in the proof and restated here for convenience.

\begin{lemma}
Let $\nu$ be a measure on $\square_{R}$ such that $\mathcal{E}(\nu) < \infty$ and $\nu \in L^{\infty}$, and $i_{N}$ be such that
\begin{equation}
  \lim_{N \to \infty}  \frac{i_{N}}{N^{1 - \lambda d}} = |\nu|.
\end{equation}
Then
\begin{equation}
    \lim_{\epsilon \to 0} \lim_{N \to \infty} \left( \mathbf{T}^{N}_{\lambda}(\nu) - \inf_{\mu \in  {B}(\nu, \epsilon)  \cap \mathcal{A}_{i_{N}}^{N^{\lambda d -1}}(\square_{R})}  \mathbf{T}^{N, \neq }_{\lambda}(\mu) \right) =0.
\end{equation}
\end{lemma}

\begin{proof}
Let $\mu \in  {B}(\nu, \epsilon)  \cap \mathcal{A}_{i_{N}}^{N^{\lambda d -1}}(\square_{R})$, and let 
\begin{equation}
    \mu^{*} = \mu \ast \lambda_{N^{\lambda - \frac{1}{d}}},
\end{equation}
(see Remark \ref{def:lamb} for notation). We claim that 
\begin{equation}
    \mathbf{T}^{N}_{\lambda}(\mu^{*}) \leq  \mathbf{T}^{N, \neq }_{\lambda}(\mu) + C N^{2\left( \lambda - \frac{1}{d} \right)},
\end{equation}
where $C$ depends on $\nu$. To see this, let
\begin{equation}
    \rho^{*}=\\
    \operatorname{argmin}_{\rho \in \mathcal{M}^{+} ( \mathbf{R}^{d} \setminus \square_{R} )} \Bigg( \frac{1}{2}\mathcal{E}^{\neq}_{\square_{R}}\left( \mu + \rho - \mu_{\beta}^{N^{\lambda}} \right) -\int_{\mathbf{R}^{d}} \log\left( \mu_{\beta}^{N^{\lambda}} \right) d\rho + \mbox{ent}[\rho] \Bigg),
\end{equation}
where the minimum is taken over $\rho$ such that
\begin{equation}
    \int_{\mathbf{R}^{d}} \mu + \rho - \mu_{\beta}^{N^{\lambda}}=0.
\end{equation}

Then we can use $\rho^{*}$ as a test function in the definition of $\mathbf{T}^{N}_{\lambda}(\mu^{*})$ and get 
\begin{equation}
     \mathbf{T}^{N}_{\lambda}(\mu^{*}) \leq \frac{1}{2}\mathcal{E}\left( \mu^{*} + \rho^{*} - \mu_{\beta}^{N^{\lambda}} \right) -\int_{\mathbf{R}^{d}} \log\left( \mu_{\beta}^{N^{\lambda}} \right) d\rho^{*} + \mbox{ent}[\rho^{*}].
\end{equation}

Using Lemmas \ref{smearinglemma1}, \ref{smearinglemma2}, \ref{smearinglemma3}, \ref{smearinglemma4}, we have that
\begin{equation}
    \mathcal{E}\left( \mu^{*} + \rho^{*} - \mu_{\beta}^{N^{\lambda}} \right) \leq \mathcal{E}^{\neq}_{\square_{R}}\left( \mu + \rho^{*} - \mu_{\beta}^{N^{\lambda}} \right) + C N^{2\left( \lambda - \frac{1}{d} \right)},
\end{equation}
where $C$ depends on $\nu$. We, therefore, get that  
\begin{equation}
    \mathbf{T}^{N}_{\lambda}(\mu^{*}) \leq  \mathbf{T}^{N, \neq }_{\lambda}(\mu) + C N^{2\left( \lambda - \frac{1}{d} \right)},
\end{equation}
where $C$ depends on $\nu$. Note that 
\begin{equation}
    \lim_{N \to \infty} \| \mu - \mu^{*} \|_{BL} =0,
\end{equation}
therefore we are left with proving that
\begin{equation}
    \lim_{\epsilon \to 0} \lim_{N \to \infty} \left( \mathbf{T}^{N}_{\lambda}(\nu) - \inf_{\mu \in  {B}(\nu, \epsilon) } \left( \mathbf{T}^{N}_{\lambda}(\mu) \right) \right) =0.
\end{equation}

To see this, let 
\begin{equation}
   \mu_{N}^{\epsilon} =  \operatorname{argmin}_{\mu \in  {B}(\nu, \epsilon) }  \mathbf{T}^{N}_{\lambda}(\mu).
\end{equation}
We assume that the infimum is achieved for clarity of exposition. Otherwise, we would repeat the argument up to an arbitrarily small error. Let
\begin{equation}
    \rho^{N}_{\epsilon}=\\
    \operatorname{argmin}_{\rho \in \mathcal{M}^{+} ( \mathbf{R}^{d} \setminus \square_{R})} \Bigg( \frac{1}{2}\mathcal{E}\left( \mu_{N}^{\epsilon} + \rho - \mu_{\beta}^{N^{\lambda}} \right) -\int_{\mathbf{R}^{d}} \log\left( \mu_{\beta}^{N^{\lambda}} \right) d\rho + \mbox{ent}[\rho] \Bigg),
\end{equation}
where the minimum is taken over $\rho$ such that
\begin{equation}
    \int_{\mathbf{R}^{d}} \mu_{N}^{\epsilon} + \rho - \mu_{\beta}^{N^{\lambda}}=0.
\end{equation}

Then we can use $\rho^{N}_{\epsilon}$ as a test function in the definition of $\mathbf{T}^{N}_{\lambda}(\nu) $ and get
\begin{equation}
\begin{split}
    &\lim_{\epsilon \to 0} \lim_{N \to \infty} \left( \mathbf{T}^{N}_{\lambda}(\nu) - \inf_{\mu \in  {B}(\nu, \epsilon) } \left( \mathbf{T}^{N}_{\lambda}(\mu) \right) \right) \leq \\
    &\lim_{\epsilon \to 0} \lim_{N \to \infty} \frac{1}{2}\mathcal{E}\left( \nu + \rho^{N}_{\epsilon} - \mu_{\beta}^{N^{\lambda}} \right) - \frac{1}{2}\mathcal{E}\left( \mu_{N}^{\epsilon} + \rho^{N}_{\epsilon} - \mu_{\beta}^{N^{\lambda}} \right) =\\
    &  \lim_{\epsilon \to 0} \lim_{N \to \infty} \frac{1}{2}\mathcal{E}\left( \nu \right) - \frac{1}{2}\mathcal{E}\left( \mu_{N}^{\epsilon}\right) + \mathcal{G} \left( \nu - \mu_{N}^{\epsilon}, \rho^{N}_{\epsilon} - \mu_{\beta}^{N^{\lambda}}\right).
\end{split}    
\end{equation}

Note that as $\epsilon$ tends to $0$ and $N$ tends to $\infty$, $\mu_{N}^{\epsilon}$ converges weakly to $\nu$, therefore
\begin{equation}
    \lim_{\epsilon \to 0} \lim_{N \to \infty} \mathcal{E}\left( \nu \right) - \mathcal{E}\left( \mu_{N}^{\epsilon}\right) \leq 0,
\end{equation}
and 
\begin{equation}
    \lim_{\epsilon \to 0} \lim_{N \to \infty} \mathcal{G} \left( \nu - \mu_{N}^{\epsilon}, \rho^{N}_{\epsilon} - \mu_{\beta}^{N^{\lambda}}\right) =0.
\end{equation}
This implies that
\begin{equation}
    \lim_{\epsilon \to 0} \lim_{N \to \infty} \left( \mathbf{T}^{N}_{\lambda}(\nu) - \inf_{\mu \in  {B}(\nu, \epsilon) } \left( \mathbf{T}^{N}_{\lambda}(\mu) \right) \right) \leq 0,
\end{equation}
and since clearly 
\begin{equation}
    \lim_{\epsilon \to 0} \lim_{N \to \infty} \left( \mathbf{T}^{N}_{\lambda}(\nu) - \inf_{\mu \in  {B}(\nu, \epsilon) } \left( \mathbf{T}^{N}_{\lambda}(\mu) \right) \right) \geq 0,
\end{equation}
we conclude that
\begin{equation}
    \lim_{\epsilon \to 0} \lim_{N \to \infty} \left( \mathbf{T}^{N}_{\lambda}(\nu) - \inf_{\mu \in  {B}(\nu, \epsilon)  \cap \mathcal{A}_{i_{N}}^{N^{\lambda d -1}}(\square_{R})} \left( \mathbf{T}^{N, \neq }_{\lambda}(\mu) \right) \right) =0.
\end{equation}

\end{proof}

\section{Appendix A}

In this appendix, we prove some fundamental properties of the smearing technique and energy minimizers. Loosely speaking, the smearing technique consists in studying properties about ${\rm emp}_{N}$ by analyzing instead the more regular measure ${\rm emp}_{N} \ast \lambda_{\epsilon},$ where $\lambda_{\epsilon}$ is a measure that approximates a Dirac delta on a scale $\epsilon$.

We start by recalling a few facts about smearing and electric energy. These are standard and can be found, for example, in \cite{chafai2018concentration}, \cite{leble2017large1}, or \cite{rougerie2016higher}. The proof uses that $g$ is superharmonic in its domain, and harmonic away from $0.$

\begin{lemma}\label{smearinglemma1}
For every $x \in \mathbf{R}^{d}$ and $\epsilon>0$ we have that 
\begin{equation}\label{superharmoniceq1}
   \int_{\mathbf{R}^{d}} g(x+u) \, d\lambda_{\epsilon}(u)  \leq g(x)
\end{equation}
and also that 
\begin{equation}\label{superharmoniceq2}
    \iint_{\mathbf{R}^{d} \times \mathbf{R}^{d}} g(x+u-v) \, d\lambda_{\epsilon}(u) \, d \lambda_{\epsilon}(v) \leq g(x),
\end{equation}
(see Remark \ref{def:lamb} for notation). Furthermore, eqs \eqref{superharmoniceq1} and \eqref{superharmoniceq2} become an equality if $|x|>\epsilon.$
\end{lemma}
The next lemma can also be found in \cite{chafai2018concentration} (or verified by direct computation).

\begin{lemma}\label{smearinglemma2}
Let $\epsilon>0$, then for $d \geq 3$,
\begin{equation}
    \mathcal{E}(\lambda_{\epsilon})=g(\epsilon)    \mathcal{E}(\lambda_{1}).
\end{equation}
For $d=2$,
\begin{equation}
     \mathcal{E}(\lambda_{\epsilon})=g(\epsilon)+\mathcal{E}(\lambda_{1}).
\end{equation}
\end{lemma}

\begin{lemma}\label{smearinglemma3}
Let $\left\{x_{i} \right\}_{i=1}^{N} \in \mathbf{R}^{d},$  let $\phi=\frac{1}{N}\sum_{i=1}^{N} \delta_{x_{i}}$ and $\phi_{\epsilon}=\phi \ast \lambda_{\epsilon}.$ Then
\begin{equation}\label{approximatingenergybysmoothfunctions}
    \frac{1}{N^{2}}\sum_{i\neq j} g(x_{i}-x_{j}) \geq \mathcal{E}\left( \phi_{\epsilon} \right)-\frac{1}{N}g(\epsilon)\mathcal{E}(\lambda_{1}).
\end{equation}
Furthermore, eq. \eqref{approximatingenergybysmoothfunctions} is an equality if $\epsilon \leq \min\left\{ |x_{i}-x_{j}|\right\}.$
\end{lemma}

\begin{lemma}\label{smearinglemma4}
Let $\phi=\frac{1}{N}\sum_{i=1}^{N} \delta_{x_{i}}$ for $\left\{ x_{i} \right\}_{i=1}^{N} \in \mathbf{R}^{d}.$ Let $\phi_{\epsilon}=\phi \ast \lambda_{\epsilon}$ for $\epsilon>0.$ Let $\mu$ be a measure with an $L^{\infty}$ density. Then there exists $C>0,$ which depends only on $\|\mu \|_{L^{\infty}}$ such that
\begin{equation}\label{smearinglemma}
       \left| \mathcal{G}(P_{\epsilon}, \mu)  - \mathcal{G}(P, \mu) \right| \leq C \epsilon^{2}.
\end{equation}
\end{lemma}

\section{Appendix B}
\label{App:construction}

We will now prove Lemma \ref{goodenergygoodvolume}, which we restate here for convenience.

\begin{lemma}\label{goodenergygoodvolume2}
Let $\mu_{n}, \overline{\nu}$ be probability measures on a compact set $\Omega$ such that 
\begin{equation}
   \limsup_{n \to \infty} {\rm ent}[\mu_{n}]< \infty, \quad  {\rm ent}[\overline{\nu}]< \infty,
\end{equation}
$\overline{\nu} \in L^{\infty}(\Omega)$, and 
\begin{equation}
    \mathcal{E}(\overline{\nu}) < \infty.
\end{equation}
Assume that $\mu_{n}(x)$ is uniformly equi-continuous and bounded away from $0$ uniformly in $x$ and $n$.
Then for every $\epsilon, \delta, \eta, $ there exists a family of configurations
\begin{equation}
    \Lambda_{\delta}^{\eta, \epsilon} \subset \mathbf{R}^{d \times n}
\end{equation}
such that
\begin{itemize}
    \item[$\bullet$] \begin{equation}
    {\rm emp}_{n}(X_{n}) \in B(\overline{\nu}, \epsilon)
\end{equation}
for any $X_{n} \in  \Lambda_{\delta}^{\eta, \epsilon}.$
 \item[$\bullet$] 
 \begin{equation}
     \liminf_{n \to \infty} \frac{1}{n} \log \left( \int_{X_{n} \in  \Lambda_{\delta}^{\eta, \epsilon}} \Pi_{i=1}^{n} \mu_{n}(x_{i}) dX_{n} \right) \geq - \liminf_{n \to \infty}{\rm ent}[\overline{\nu}|\mu_{n}] - \delta. 
 \end{equation}
  \item[$\bullet$] 
  \begin{equation}\label{goodenergy2}
     \limsup_{n \to \infty} \left| \mathcal{E}^{\neq} ({\rm emp}_{n}(X_{n}) - \overline{\nu} )  \right| \leq \eta^{2}. 
  \end{equation}
  \item There exists $r>0$ such that
  \begin{equation}
      d(x_{i}, \partial \Omega) > r n^{-\frac{1}{d}} \quad {\rm and} \quad d(x_{i}, x_{j})> r n^{-\frac{1}{d}},
  \end{equation}
  for $i \neq j$.
\end{itemize}

\end{lemma}
\begin{proof}

\textbf{\textit{Step 1}}: Definition

First, we subdivide $\Omega$ into cubes $K_{j}$ of size $\overline{\eta}>0$ and center $x_{j},$ for $\overline{\eta}>0$ to be determined later. 

Let either
\begin{equation}
    n_{j} = \ceil*{n \overline{\nu}(K_{j})}
\end{equation}
or 
\begin{equation}
    n_{j} = \floor*{n \overline{\nu}(K_{j})},
\end{equation}
chosen so that 
\begin{equation}
    \sum_{j} n_{j} = n.
\end{equation}

The procedure for determining the point configuration of $n_{j}$ points is: $y_{1}$ is chosen at random from $K_{j}^{\tau},$ where $K_{j}^{\tau}$ is the cube $K_{j}$ minus a boundary layer of width $\tau$, $y_{2}$ is chosen at random from 
\begin{equation}
    K_{j}^{\tau} \setminus B(y_{1}, \tau).
\end{equation}
Then, for $i = 1...n_{j},$ the point $y_{i}$ is chosen at random from
\begin{equation}
    K_{j}^{\tau} \setminus \bigcup_{l=1}^{i-1} B(y_{l}, \tau).
\end{equation}

In other words, 
\begin{equation}
     \Lambda_{\delta}^{\eta, \epsilon} = \bigcup_{\sigma \in \text{sym}[1:n]} \bigotimes_{j} \bigotimes_{i=1}^{n_{j}} \left( K_{j}^{\tau} \setminus \bigcup_{ l=1}^{i-1} B(y_{\sigma(l)}, \tau) \right).
\end{equation}

We set $\tau = \alpha \overline{\eta} n_{j}^{-\frac{1}{d}},$ for some $\alpha \in (0,1)$ to be determined later. For $\alpha$ small enough, the procedure is well defined, in the sense that it is possible to choose $n_{j}$ points in this way. 

For $\overline{\eta}$ small enough, any $X_{n} \in \Lambda_{\delta}^{\eta, \epsilon} $ satisfies 
\begin{equation}
    {\rm emp}_{n}(X_{n}) \in B(\overline{\nu}, \epsilon).
\end{equation}

We immediately get that $d(x_{i}, \partial \Omega) > r n^{-\frac{1}{d}}, d(x_{i}, x_{j})> r n^{-\frac{1}{d}}$ for some $r>0$. We now prove that these configurations have the right volume and energy.

\textbf{\textit{Step 2}}: Volume Estimate

To give intuition, we first treat the case $\mu_{n}$ as  the uniform measure on $\Omega$. In this case, we have
\begin{equation}
    \begin{split}
        \mu_{n}^{\otimes n} (\Lambda_{\delta}^{\eta, \epsilon} ) &= \frac{n!}{\Pi_{i} n_{i}!} \Pi_{j} \Pi_{p=1}^{n_{j}-1} (\overline{\eta}^{d} - k_{d}\overline{\eta}^{d-1}\tau - c_{d} p \tau^{d}) \\
        &= \frac{n!}{\Pi_{i} n_{i}!} \Pi_{j} \overline{\eta}^{d n_{j}} \Pi_{p=1}^{n_{j}-1} (1 - \frac{\tau k_{d}}{\overline{\eta}} - \frac{c_{d} p \alpha^{d}}{n_{j}}), 
    \end{split}
\end{equation}
where $c_{d}, k_{d}$ are constants which depend only on $d$. On the other hand, the volume of all configurations with exactly $n_{j}$ points in cube $K_{j}$ is given by
\begin{equation}
   \frac{n!}{\Pi_{i} n_{i}!} \Pi_{j} \overline{\eta}^{d n_{j}}.
\end{equation}

By Sanov's theorem, we have that 
\begin{equation}
    \frac{n!}{\Pi_{i} n_{i}!} \Pi_{j} \overline{\eta}^{d n_{j}} = \exp(-n [{\rm ent}[\overline{\nu}| \mu_{n}] + o_{n}(1)]).
\end{equation}

For a general $\mu_{n},$ we have that the volume of all configurations with exactly $n_{j}$ points in cube $K_{j}$ is given by
\begin{equation}
    \frac{n!}{\Pi_{i} n_{i}!} \Pi_{j} [\mu_{n}(K_{j})]^{n_{j}},
\end{equation}
and that by Sanov's theorem
\begin{equation}
    \frac{n!}{\Pi_{i} n_{i}!} \Pi_{j} [\mu_{n}(K_{j})]^{n_{j}} = \exp(-n [{\rm ent}[\overline{\nu}| \mu_{n}] + o_{n}(1)]).
\end{equation}

On the other hand, we can estimate 
\begin{equation}
    \begin{split}
        \log \left( \Pi_{j}  \Pi_{p=1}^{n_{j}-1} (1 - \frac{ k_{d} \tau}{\overline{\eta}} - \frac{c_{d} p \alpha^{d}}{n_{j}})   \right) &=  \sum_{j}  \sum_{p=1}^{n_{j}-1} \log \left(1 - \frac{ k_{d}\tau}{\overline{\eta}} - \frac{c_{d} p \alpha^{d}}{n_{j}}\right) \\
        &\leq  \alpha  k_{d} \sum_{j} n_{j}^{1-\frac{1}{d}} + c_{d} \alpha^{d} \sum_{j} n_{j} \\
        & \leq C \alpha n,
    \end{split}
\end{equation}

where $C$ depends on $\overline{\nu}$. Using the hypothesis that $\mu_{n}$ is uniformly equi-continuous, we have that for any any $\overline{\delta} > 0$ there exists $\overline{\eta}^{*}$ such that if $\overline{\eta} < \overline{\eta}^{*}$ we have
\begin{equation}
    \frac{\mu_{n}(x)}{\mu_{n}(y)} \in (1-\overline{\delta}, 1+\overline{\delta})
\end{equation}
for any $x,y \in \Omega$.

Hence, we have 
\begin{equation}
    \begin{split}
        &\log \left( \mu_{n}^{\otimes n} (\Lambda_{\delta}^{\eta, \epsilon} ) \right) \geq \\
        &\log\left( \frac{n!}{\Pi_{i} n_{i}!} \Pi_{j} [\mu_{n}(K_{j})]^{n_{j}} \right) - \log \left( \Pi_{j}  \Pi_{p=1}^{n_{j}-1} (1 - \frac{\tau}{\overline{\eta}} - \frac{c_{d} p \alpha^{d}}{n_{j}})   \right) - o_{n}(n) =\\
        & -n \left({\rm ent}[\overline{\nu}| \mu_{n}] - C \alpha - o_{n}(1) - o_{\overline{\eta}}(1) \right).
    \end{split}
\end{equation}

\textbf{\textit{Step 3}}: Energy Estimate

The idea for the energy estimate will be to prove that
\begin{equation}
    h^{{\rm emp}_{n}-\overline{\nu}}
\end{equation}
is typically pointwise small. Then the smallness of the energy will be a consequence of the finite mass of the measures $\overline{\nu}$ and ${\rm emp}_{n}.$

Let $x \in K_{i}.$ Then we can write
\begin{equation}\label{breakdownoftheenergy}
     h^{{\rm emp}_{n}-\overline{\nu}} (x) = \int_{K_{i}} g(x-y) d ({\rm emp}_{n}-\overline{\nu})(y) + \sum_{j \neq i} \int_{  K_{j}} g(x-y) d ({\rm emp}_{n}-\overline{\nu})(y).  
\end{equation}

For any $j \neq i,$ note that the minimum distance from $x$ to $K_{j}$ is given by $|x - x_{i}| -c \overline{\eta}$ and the maximum distance from $x$ to $K_{i}$ is given by $|x - x_{i}| + c \overline{\eta}$, for some $c$ which depends on $d$ and $x$. For the rest of the proof, we assume w.l.o.g. that
\begin{equation}
    {\rm emp}_{n}(K_{j}) \geq \overline{\nu} (K_{j}),
\end{equation}
then
\begin{equation}
    \begin{split}
      &\left|  \int_{  K_{j}} g(x-y) d ({\rm emp}_{n}-\overline{\nu})(y) \right| \leq \\
      &\left| \frac{{\rm emp}_{n}(K_{i})}{(|x - x_{i}| -c \overline{\eta})^{d-2}}- \frac{\overline{\nu}(K_{i})}{(|x - x_{i}| +c \overline{\eta})^{d-2}} \right| =\\
        & \left| \frac{{\rm emp}_{n}(K_{i})}{(|x - x_{i}| -c \overline{\eta})^{d-2}}- \frac{\overline{\nu}(K_{i})}{(|x - x_{i}| -c \overline{\eta})^{d-2}} + \frac{\overline{\nu}(K_{i})}{(|x - x_{i}| -c \overline{\eta})^{d-2}} - \frac{\overline{\nu}(K_{i})}{(|x - x_{i}| +c \overline{\eta})^{d-2}} \right| \leq \\
        & \left| \frac{({\rm emp}_{n} - \overline{\nu})(K_{i})}{(|x - x_{i}| -c \overline{\eta})^{d-2}} \right| + \left|  C \overline{\eta} \frac{ \overline{\nu}(K_{i})}{|x - x_{i}| ^{d-1}} \right|,
    \end{split}
\end{equation}
for some absolute constant $C$. 

Using the hypothesis that $\overline{\nu} \in L^{\infty}(\Omega)$ we get
\begin{equation}
    |({\rm emp}_{n} - \overline{\nu})(K_{j})| \leq \frac{C}{n},
\end{equation}
where $C$ depends on $\| \overline{\nu} \|_{L^{\infty}}$. Since $\frac{1}{|x|^{d-2}}$ is integrable at the origin and $\Omega$ is compact, we have 
\begin{equation}
     \sum_{ j \neq i} \left| \frac{({\rm emp}_{n} - \overline{\nu})(K_{i})}{(|x - x_{i}| -c \overline{\eta})^{d-2}} \right| \leq  \frac{C}{n \overline{\eta}^{d}},
\end{equation}
where $C$ depends on $\| \overline{\nu} \|_{L^{\infty}}$ and $\Omega$.

Using again the hypothesis that $\overline{\nu} \in L^{\infty}$ we have 
\begin{equation}
    \begin{split}
        \sum_{ j \neq i}  \left|   \overline{\eta} \frac{ \overline{\nu}(K_{i})}{|x - x_{i}|^{d-1}} \right| &\leq C\overline{\eta} \int_{\Omega} \frac{1}{|x|^{d-1}} d, \\
         & \leq C \overline{\eta},
    \end{split}
\end{equation}
where $C$ depends on $\| \overline{\nu} \|_{L^{\infty}}$ and $\Omega$.

For the second term in equation \eqref{breakdownoftheenergy} term, we will instead work with 
\begin{equation}
    {\rm emp}_{n}^{*} = {\rm emp}_{n} \ast \lambda_{\frac{\tau}{2}},
\end{equation}
where $\tau = \alpha \overline{\eta} n_{j}^{-\frac{1}{d}},$ for some $\alpha \in (0,1)$ to be determined later (see Remark \ref{def:lamb} for notation). Note that by Lemma \ref{smearinglemma1}, and because $d(x_{i}, x_{j}) \geq r n^{-\frac{1}{d}}$ we have
\begin{equation}
     \sum_{j \neq i} \int_{  K_{j}} g(x-y) d ({\rm emp}_{n}-\overline{\nu})(y) =  \sum_{j \neq i} \int_{  K_{j}} g(x-y) d ({\rm emp}_{n}^{*}-\overline{\nu})(y).
\end{equation}
Note also that
\begin{equation}
    \| {\rm emp}_{n}^{*} \|_{L^{\infty}} \leq c_{\alpha, \overline{\nu}},
\end{equation}
where $c_{\alpha, \overline{\nu}}$ is a constant that depends on $\alpha$ and $\| \overline{\nu} \|_{L^{\infty}}$. Hence
\begin{equation}
    \begin{split}
        \left| \int_{K_{i}} g(x-y) d ({\rm emp}_{n}^{*} - \overline{\nu})(y) \right| & \leq c_{\alpha, \overline{\nu}} \int_{K_{i}} \frac{1}{|x|^{d-2}} dx \\
        &\leq c_{\alpha, \overline{\nu}} \overline{\eta}^{2},
    \end{split}
\end{equation}
where $c_{\alpha, \overline{\nu}}$ is a (new) constant that depends on $\alpha$ and $\| \overline{\nu} \|_{L^{\infty}}$.

Putting everything together, we get 
\begin{equation}
    |h^{{\rm emp}^{*}_{n}-\overline{\nu}}| \leq \frac{C}{n \overline{\eta}^{d}} + C \overline{\eta} + c_{\alpha, \overline{\nu}} \overline{\eta}^{2},
\end{equation}
where $C$ depends on $\overline{\nu}$ and $\Omega$ and $c_{\alpha, \overline{\nu}}$ depends, in addition, on $\alpha$. 

Hence
\begin{equation}
    \begin{split}
        &\iint_{\Omega \times \Omega} g(x-y) d ({\rm emp}_{n}^{*} - \overline{\nu})(x) d ({\rm emp}_{n}^{*} - \overline{\nu})(y) \leq \\
        & \| h^{{\rm emp}^{*}_{n}-\overline{\nu}}\|_{L^{\infty}} \|{\rm emp}_{n}^{*} - \overline{\nu}\|_{TV} \leq \\
        & \frac{C}{n \overline{\eta}^{d}} + C \overline{\eta} + c_{\alpha, \overline{\nu}} \overline{\eta}^{2}.
    \end{split}
\end{equation}

Making $\overline{\eta}$ small enough after having chosen $\alpha,$ while keeping $\overline{\eta} >> n^{-\frac{1}{d}},$ we have that for any $\eta >0$ we can find parameters such that
\begin{equation}
     \limsup_{n \to \infty} \left| \mathcal{E}({\rm emp}_{n}^{*}(X_{n}) - \overline{\nu} )  \right| \leq \eta^{2}, 
\end{equation}
  which implies that
\begin{equation}
     \limsup_{n \to \infty} \left| \mathcal{E}^{\neq} ({\rm emp}_{n}(X_{n}) - \overline{\nu} )  \right| \leq \eta^{2}. 
\end{equation}
\end{proof}

%

\section{Acknowledgements}

I thank Sylvia Serfaty for her guidance during this project. I thank Ofer Zeitouni and Thomas Leblé for useful conversations.

\bibliographystyle{siam}
\bibliography{bibliography.bib}

\begin{thebibliography}{10}

\bibitem{armstrong2019local}
{\sc S.~Armstrong and S.~Serfaty}, {\em Local laws and rigidity for {C}oulomb
  gases at any temperature}, Annals of Probability, in press-arXiv preprint
  arXiv:1906.09848,  (2019).

\bibitem{armstrong2019thermal}
\leavevmode\vrule height 2pt depth -1.6pt width 23pt, {\em Thermal
  approximation of the equilibrium measure and obstacle problem}, arXiv
  preprint arXiv:1912.13018,  (2019).

\bibitem{arous1997large}
{\sc G.~B. Arous and A.~Guionnet}, {\em Large deviations for wigner's law and
  voiculescu's non-commutative entropy}, Probability theory and related fields,
  108 (1997), pp.~517--542.

\bibitem{bauerschmidt2017local}
{\sc R.~Bauerschmidt, P.~Bourgade, M.~Nikula, and H.-T. Yau}, {\em Local
  density for two-dimensional one-component plasma}, Communications in
  Mathematical Physics, 356 (2017), pp.~189--230.

\bibitem{bauerschmidt2019two}
\leavevmode\vrule height 2pt depth -1.6pt width 23pt, {\em The two-dimensional
  coulomb plasma: quasi-free approximation and central limit theorem}, Advances
  in Theoretical and Mathematical Physics, 23 (2019), pp.~841--1002.

\bibitem{bekerman2018clt}
{\sc F.~Bekerman, T.~Lebl{\'e}, S.~Serfaty, et~al.}, {\em Clt for fluctuations
  of $\beta$-ensembles with general potential}, Electronic Journal of
  Probability, 23 (2018).

\bibitem{ben1998large}
{\sc G.~Ben~Arous and O.~Zeitouni}, {\em Large deviations from the circular
  law}, ESAIM: Probability and Statistics, 2 (1998), pp.~123--134.

\bibitem{bodineau1999stationary}
{\sc T.~Bodineau and A.~Guionnet}, {\em About the stationary states of vortex
  systems}, Annales de l'Institut Henri Poincare (B) Probability and
  Statistics, 35 (1999), pp.~205--237.

\bibitem{borodin2019random}
{\sc A.~Borodin, I.~Corwin, and A.~Guionnet}, {\em Random matrices}, in
  IAS/Park City Mathematics Series, vol.~26, American Mathematical Soc., 2019.

\bibitem{borot2013asymptotic}
{\sc G.~Borot and A.~Guionnet}, {\em Asymptotic expansion of $\beta$ matrix
  models in the one-cut regime}, Communications in Mathematical Physics, 317
  (2013), pp.~447--483.

\bibitem{bourgade2012bulkb}
{\sc P.~Bourgade, L.~Erd{\H{o}}s, and H.-T. Yau}, {\em Bulk universality of
  general $\beta$-ensembles with non-convex potential}, Journal of mathematical
  physics, 53 (2012), p.~095221.

\bibitem{bourgade2014edge}
{\sc P.~Bourgade, L.~Erd{\"o}s, and H.-T. Yau}, {\em Edge universality of beta
  ensembles}, Communications in Mathematical Physics, 332 (2014), pp.~261--353.

\bibitem{bourgade2014universality}
{\sc P.~Bourgade, L.~Erd{\H{o}}s, H.-T. Yau, et~al.}, {\em Universality of
  general $\beta$-ensembles}, Duke Mathematical Journal, 163 (2014),
  pp.~1127--1190.

\bibitem{bourgade2014local}
{\sc P.~Bourgade, H.-T. Yau, and J.~Yin}, {\em Local circular law for random
  matrices}, Probability Theory and Related Fields, 159 (2014), pp.~545--595.

\bibitem{chafai2014first}
{\sc D.~Chafa{\"\i}, N.~Gozlan, P.-A. Zitt, et~al.}, {\em First-order global
  asymptotics for confined particles with singular pair repulsion}, The Annals
  of Applied Probability, 24 (2014), pp.~2371--2413.

\bibitem{chafai2018concentration}
{\sc D.~Chafai, A.~Hardy, and M.~Ma{\"\i}da}, {\em Concentration for coulomb
  gases and coulomb transport inequalities}, Journal of Functional Analysis,
  275 (2018), pp.~1447--1483.

\bibitem{garcia2019large}
{\sc D.~Garc{\'\i}a-Zelada}, {\em A large deviation principle for empirical
  measures on polish spaces: Application to singular gibbs measures on
  manifolds}, in Annales de l'Institut Henri Poincar{\'e}, Probabilit{\'e}s et
  Statistiques, vol.~55, Institut Henri Poincar{\'e}, 2019, pp.~1377--1401.

\bibitem{hardin2018large}
{\sc D.~P. Hardin, T.~Lebl{\'e}, E.~B. Saff, and S.~Serfaty}, {\em Large
  deviation principles for hypersingular riesz gases}, Constructive
  Approximation, 48 (2018), pp.~61--100.

\bibitem{hardy2021clt}
{\sc A.~Hardy and G.~Lambert}, {\em Clt for circular beta-ensembles at high
  temperature}, Journal of Functional Analysis, 280 (2021), p.~108869.

\bibitem{johansson1998fluctuations}
{\sc K.~Johansson et~al.}, {\em On fluctuations of eigenvalues of random
  hermitian matrices}, Duke mathematical journal, 91 (1998), pp.~151--204.

\bibitem{lambert2019quantitative}
{\sc G.~Lambert, M.~Ledoux, C.~Webb, et~al.}, {\em Quantitative normal
  approximation of linear statistics of $\beta $-ensembles}, The Annals of
  Probability, 47 (2019), pp.~2619--2685.

\bibitem{leble2017large1}
{\sc T.~Lebl{\'e} and S.~Serfaty}, {\em Large deviation principle for empirical
  fields of log and riesz gases}, Inventiones mathematicae, 210 (2017),
  pp.~645--757.

\bibitem{leble2018fluctuations}
\leavevmode\vrule height 2pt depth -1.6pt width 23pt, {\em Fluctuations of two
  dimensional coulomb gases}, Geometric and Functional Analysis, 28 (2018),
  pp.~443--508.

\bibitem{leble2017large2}
{\sc T.~Lebl{\'e}, S.~Serfaty, and O.~Zeitouni}, {\em Large deviations for the
  two-dimensional two-component plasma}, Communications in Mathematical
  Physics, 350 (2017), pp.~301--360.

\bibitem{padilla2023concentration}
{\sc D.~Padilla-Garza}, {\em Concentration inequality around the thermal
  equilibrium measure of coulomb gases}, Journal of Functional Analysis, 284
  (2023), p.~109733.

\bibitem{petz1998logarithmic}
{\sc D.~Petz and F.~Hiai}, {\em Logarithmic energy as an entropy functional},
  Contemporary Mathematics, 217 (1998), pp.~205--221.

\bibitem{rougerie2016higher}
{\sc N.~Rougerie and S.~Serfaty}, {\em Higher-dimensional coulomb gases and
  renormalized energy functionals}, Communications on Pure and Applied
  Mathematics, 69 (2016), pp.~519--605.

\bibitem{sandier20152d}
{\sc E.~Sandier, S.~Serfaty, et~al.}, {\em 2d coulomb gases and the
  renormalized energy}, The Annals of Probability, 43 (2015), pp.~2026--2083.

\bibitem{serfaty2017microscopic}
{\sc S.~Serfaty}, {\em Microscopic description of log and coulomb gases}, arXiv
  preprint arXiv:1709.04089,  (2017).

\bibitem{serfaty2018systems}
\leavevmode\vrule height 2pt depth -1.6pt width 23pt, {\em Systems of points
  with coulomb interactions}, in Proceedings of the International Congress of
  Mathematicians: Rio de Janeiro 2018, World Scientific, 2018, pp.~935--977.

\bibitem{serfaty2020gaussian}
\leavevmode\vrule height 2pt depth -1.6pt width 23pt, {\em Gaussian
  fluctuations and free energy expansion for 2d and 3d coulomb gases at any
  temperature}, arXiv preprint arXiv:2003.11704,  (2020).

\bibitem{shcherbina2013fluctuations}
{\sc M.~Shcherbina}, {\em Fluctuations of linear eigenvalue statistics of
  $\beta$ matrix models in the multi-cut regime}, Journal of Statistical
  Physics, 151 (2013), pp.~1004--1034.

\end{thebibliography}

\end{document}